\newif\ifarxiv
\arxivtrue
\documentclass{article}

\ifarxiv
\usepackage[accepted, nohyperref]{icml2018
}
\else
\usepackage{icml2018}
\fi

\usepackage{graphicx}
\graphicspath{{./fig/}}
\usepackage{xcolor, color, colortbl}
\usepackage{url}
\usepackage{commath}
\usepackage{amsfonts}
\usepackage{mdframed}
\usepackage{enumerate}
\usepackage{bold-extra}
\usepackage{nccmath}
\usepackage{caption}
\usepackage{dsfont}
\usepackage{amsmath, amsthm, amssymb}
\usepackage{enumitem}
\usepackage{booktabs}
\usepackage{multirow}
\usepackage{rotating}
\usepackage{mathtools}
\usepackage{empheq}
\usepackage{nicefrac}
\usepackage{footnote}
\makesavenoteenv{tabular}
\makeatletter
\newcommand\footnoteref[1]{\protected@xdef\@thefnmark{\ref{#1}}\@footnotemark}
\makeatother
\usepackage{tikz}

\definecolor{mydarkblue}{rgb}{0,0.08,0.45}
\usepackage[colorlinks=true,
    linkcolor=mydarkblue,
    citecolor=mydarkblue,
    filecolor=mydarkblue,
    urlcolor=mydarkblue]{hyperref}
\definecolor{myblue}{HTML}{D2E4FC}
\newcommand*\mybluebox[1]{%
\colorbox{myblue}{\hspace{1em}#1\hspace{1em}}}

\DeclareMathOperator*{\argmin}{\mathbf{arg\,min}}
\DeclareMathOperator*{\dom}{dom}

\DeclareMathOperator*{\minimize}{minimize}

\DeclareMathOperator*{\Fix}{{Fix}}

\DeclareMathOperator{\prox}{\mathbf{prox}}

\usepackage{pifont}
\newcommand{\cmark}{\ding{51}}%
\newcommand{\xmark}{\ding{55}}%

\def\RR{{\mathbb R}}

\def\xx{{\boldsymbol x}}
\def\XX{{\boldsymbol X}}
\def\ZZ{{\boldsymbol Z}}
\def\YY{{\boldsymbol Y}}
\def\uu{{\boldsymbol u}}
\def\UU{{\boldsymbol U}}
\def\rr{{\boldsymbol r}}

\def\yy{{\boldsymbol y}}

\def\zz{{\boldsymbol z}}
\def\KK{{\boldsymbol K}}
\def\CC{{\boldsymbol C}}
\def\TT{{\boldsymbol T}}
\def\balpha{{\boldsymbol u}}
\def\defas{\stackrel{\text{def}}{=}}

\newtheorem{theorem}{Theorem}
\newtheorem{lemma}{Lemma}

\newtheorem{corollary}{Corollary}
\newtheorem{definition}{Definition}

\definecolor{mydarkblue}{rgb}{0,0.08,0.45}

\usepackage{hyperref}
\hypersetup{colorlinks=true,
    linkcolor=mydarkblue,
    citecolor=mydarkblue,
    filecolor=mydarkblue,
    urlcolor=mydarkblue}

\newtheorem{innercustomtheorem}{Theorem}
\newenvironment{customtheorem}[1]
  {\renewcommand\theinnercustomtheorem{#1}\innercustomtheorem}
  {\endinnercustomtheorem}

\newtheorem{innercustomcorollary}{Corollary}
\newenvironment{customcorollary}[1]
  {\renewcommand\theinnercustomcorollary{#1}\innercustomcorollary}
  {\endinnercustomcorollary}

\definecolor{Gray}{gray}{0.92}

\usepackage[linesnumbered, ruled,vlined]{algorithm2e}
\SetKwBlock{Repeat}{repeat}{}
\SetKwBlock{Variantone}{Variant 1}{}
\SetKwBlock{Varianttwo}{Variant 2}{}

\usepackage{titlesec}
\usepackage{pythonhighlight}
\definecolor{mydarkred}{RGB}{192,47,25}
\newcommand{\red}{\color{mydarkred}}
\definecolor{mydarkgreen}{RGB}{39,130,67}
\newcommand{\green}{\color{mydarkgreen}}

\begin{document}
\twocolumn[
% \icmltitle{Three Operator Splitting with Line Search}
\icmltitle{Adaptive Three Operator Splitting}




\begin{icmlauthorlist}
\icmlauthor{Fabian Pedregosa}{ber,eth}
\icmlauthor{Gauthier Gidel}{um}
\end{icmlauthorlist}

\icmlaffiliation{ber}{University of California at Berkeley, USA}
\icmlaffiliation{eth}{Department of Computer Science, ETH Zurich, Switzerland}
\icmlaffiliation{um}{Mila - Universit\'e de Montr\'eal, Canada}
% \icmlaffiliation{ed}{School of Computation, University of Edenborrow, Edenborrow, United Kingdom}

\icmlcorrespondingauthor{Fabian Pedregosa}{\href{mailto:f@bianp.net}{f@bianp.net}}

\icmlkeywords{optimization, proximal splitting, line search, three operator splitting}

\vskip 0.3in

]

\printAffiliationsAndNotice{}  % leave blank if no need to mention equal
\begin{abstract}
We propose and analyze an adaptive step-size variant of the Davis-Yin three operator splitting. This method can solve optimization problems composed by a sum of a smooth term for which we have access to its gradient and an arbitrary number of potentially non-smooth terms for which we have access to their proximal operator.
The proposed method sets the step-size based on local information of the objective --hence allowing for larger step-sizes--, only requires two extra function evaluations per iteration and does not depend on any step-size hyperparameter besides an initial estimate.
We provide an iteration complexity analysis that matches the best known results for the non-adaptive variant: sublinear convergence for general convex functions and linear convergence under strong convexity of the smooth term and smoothness of one of the proximal terms.
Finally, an empirical comparison with related methods on 6 different problems illustrates the computational advantage of the proposed method.
\end{abstract}

\section{Introduction}

Minimizing the sum of a smooth and a non-smooth term is at the core of many optimization problems that arise in machine learning and signal processing~\citep{rudin1992nonlinear,candes2006robust, chambolle2016introduction}.
In a few but important cases, such as $\ell_1$ or group lasso regularization, the \mbox{non-smooth} term is simple enough so that its proximal operator is available in closed form or at least fast to compute. In this case, highly scalable methods such as proximal
gradient descent~\citep{beck2009gradient,nesterov2013gradient} or proximal coordinate descent~\citep{richtarik2014iteration} have shown state of the art performance.
However,
 the desire to model increasingly complex phenomena has led to the development of a flurry of penalties with costly to compute proximal operator. Examples are the overlapping group lasso~\citep{jacob2009group}, multidimensional total variation~\citep{barbero2014modular} or trend filtering~\citep{kim2009ell1}, to name a few.

A key observation is that, despite the difficulty in computing its proximal operator, many of these penalties can be decomposed as a sum of terms for which we have access to their proximal operator.
Proximal splitting methods like the three operator splitting~\citep{davis2017three} offer a principled way to incorporate these penalties into the optimizer. In this work we will describe a method to solve optimization problems of the form
\begin{empheq}[box=\mybluebox]{equation}\tag{OPT}\label{eq:obj_fun}
  \minimize_{\xx \in \RR^p}\,\vphantom{\sum_i^n} f(\xx) + g(\xx) + h(\xx) ~,
\end{empheq}
where $f$ is convex and $L_f$-smooth (i.e., differentiable with $L_f$-Lipschitz gradient) and $g, h$ are both convex but potentially non-smooth. We further assume $g$ and $h$ are \emph{proximal}, i.e., we have access to the proximal operator.

This formulation allows to express a broad range of problems arising in machine learning and signal processing: the smooth term includes the least squares or logistic loss functions; the two proximal terms can be extended to an arbitrary number via a product space formulation and as we will see in \S\ref{scs:learning_multiple_penalties} include many important penalties such as the group lasso with overlap, total variation, $\ell_1$ trend filtering, etc. Furthermore, the penalties can be extended-valued, thus allowing an intersection for convex constraints through the use of the indicator function.

{\bfseries The three operator splitting} (TOS) method \citep{davis2017three} is a recently proposed method for problems of the form \eqref{eq:obj_fun}.
% It is a generalization of both the Douglas-Rachford method~\citep{lions1979splitting} and the proximal gradient descent~\citep{beck2009gradient}.
%
At each iteration, it only requires to evaluate once the gradient of $f$ and the proximal operator of $g$ and $h$. It also relies on one step-size parameter, and while it can be set based on the Lipschitz constant of the gradient of $f$, this is not entirely satisfactory for two reasons. First, this constant is often costly to compute. Second, this constant is a global upper bound on the Lipschitz constant, while locally the Lipschitz constant might be smaller, allowing for larger step-sizes.

\emph{Adaptive step-size} methods, also known as inexact and backtracking line search, instead choose the step-size by verifying a sufficient decrease condition at each iteration. 
 This allows to take larger step-sizes and has proven to be an important ingredient in the practical implementation of first and second-order methods~\citep{nocedal2006numerical}.
% In this case the step-size is chosen using local information of the objective, allowing for much larger step-sizes.

\paragraph{Outline and main contributions.} Our main contribution is the development and analysis of an adaptive variant of the TOS algorithm.
%that does not require to set the step-size parameter, and instead relies on a backtracking line search strategy to estimate it along the iterates.
The proposed algorithm does not depend on any step-size hyperparameter (besides an initial estimate) and enjoys similar convergence guarantees as the non adaptive variant.
The paper is organized as follows:
\begin{itemize}[leftmargin=*]
\item \emph{Methods}. \S\ref{scs:methods} describes the proposed algorithm, extended in \S\ref{scs:extension_k_terms} to an arbitrary number of proximal terms.
\item \emph{Analysis}. \S\ref{scs:analysis} provides a convergence analysis based on an interpretation of the algorithm as a saddle-point optimization method. This significantly departs from the analysis of~\citet{davis2017three} for the non adaptive variant and results in improved and more general rates.
\item \emph{Applications}.
\S\ref{scs:applications} discusses the application to different penalties and presents an empirical comparison on 6 different problems and 5 different penalties.
\end{itemize}

\paragraph{Notation.} We denote vectors with boldface lower case letters (i.e., $\xx$), and matrices and vector-valued functions in boldface upper case (i.e., $\XX$, $\TT(\cdot)$). $\|\cdot\|$ denotes the euclidean vector norm.
%, $\|\xx\|_1 \defas \sum_i |\xx_i|$ the $\ell_1$ norm and $$.
Given a matrix $\XX \in \RR^{n \times p}$, we denote by $\overline{\XX}$ the average along rows, that is, $\overline{\XX} = \nicefrac{1}{n}\sum_{i=1}^n \XX_i$.
We make extensive use of the proximal operator, defined for a convex function $\varphi$ and $\gamma > 0$ as
\begin{equation}
\prox_{\gamma \varphi}(\xx) \defas \argmin_{\zz \in \RR^p}\Big\{ \varphi(\zz) + \mfrac{1}{2 \gamma}\|\xx - \zz\|^2\Big\}~.
\end{equation}
The domain of a function $f:\RR^p \to ]-\infty, \infty]$ is $\dom f \defas \{\xx \in \RR^p\,|\,f(\xx) < \infty \}$. The indicator function is denoted $\imath\{\text{condition}\}$, which is $0$ if condition is verified and $+\infty$ otherwise.
% The sum operator between two sets refers to the Minkowski addition, i.e., $A + B \defas \{a + b\, |\, a \in A, b \in B\}$.
Basic properties and definitions of convex functions are provided for convenience in \ref{apx:basic_definitions}.

\subsection{Related work}

Proximal splitting methods that can solve problems involving a sum of terms by accessing the proximal operators of their constituents can be traced back to the 1970s in the works of \citet{glowinski1975approximation,gabay1976dual,lions1979splitting}. There has been a surge in interest in these methods in the last years due to their applicability in machine learning~\citep{parikh2013proximal}, signal processing~\citep{combettes2011proximal} and parallel optimization~\citep{boyd2011distributed}.

% \paragraph{Related work}
Algorithms to solve problems of the form \eqref{eq:obj_fun} with two or more proximal terms and a smooth term accessed via its gradient have recently been proposed. Examples are the generalized forward-backward splitting~\citep{raguet2013generalized}, the three operator splitting (TOS)~\citep{davis2017three}, the
 primal-dual hybrid gradient (PDHG) method, proposed in \citep{condat2013primal,vu2013splitting} and analyzed by~\citet{chambolle2016ergodic} and the very recent primal-dual three operator splitting \citep{yan2018new}. We note that the last two methods can optimize a more general objective function in which $h(\xx)$ is replaced with $h(\boldsymbol{K}\xx)$ for an arbitrary matrix $\boldsymbol{K}$.
The original formulation of these methods requires to set the step-size based on criteria such as the Lipschitz constant of the gradient of the smooth term,  but variants with adaptive step-size have recently emerged.

An adaptive step-size variant of the PDHG algorithm has recently been proposed by \citet[\S5]{malitsky2016first}.
Compared to the proposed method, it requires one less function evaluation per iteration but since the original algorithm has two step-sizes, it still relies on one step-size hyperparameter. Convergence rates are not derived.

A different adaptive step-size strategy was proposed by \citet{giselsson2016line} as a general scheme for averaged operators. TOS is averaged for step-sizes $ < 2/L_f$, and we denote the combination of both methods LSAO-TOS. An $\mathcal{O}(1/\sqrt{t})$ convergence rate in terms of the operator residual norm is derived. Unfortunately, this quantity is difficult to relate to the more common objective function suboptimality used in the other contributions.

Another adaptive step-size variant of TOS was proposed without proof in the technical report \citet[Algorithm 3]{davis2015three}. It uses the same sufficient decrease inequality as our method, although the iterates are defined differently. We found the algorithm sometimes non-convergent and did not consider it further.

In contrast, we provide a convergence analysis for our method that achieves a $\mathcal{O}(1/t)$ convergence rate for the ergodic (i.e., averaged) iterate, and linear convergence under stronger assumptions, matching and in some cases even improving the best known rates of the non adaptive variant.

{
\centering
\footnotesize
\hspace*{-0.4cm}
\setlength\tabcolsep{5pt}\begin{tabular}{c c | c c c|}
\cline{2-5}
\multicolumn{1}{c|}{} & {Method} &
{Adaptive\!} & {Sublinear rate\!\!} & Linear rate  \\
\cline{2-5}

\multicolumn{1}{c|}{} & {\cellcolor{Gray} Adaptive TOS  } & \cellcolor{Gray} & \cellcolor{Gray}  &
\cellcolor{Gray} \\
\multicolumn{1}{c|}{\multirow{4}{*}}&\footnotesize(\emph{this work})
\cellcolor{Gray}&
\multirow{-2}{*}{\green\large\cmark} \cellcolor{Gray}
\cellcolor{Gray}&
\cellcolor{Gray} \multirow{-2}{*}{\green\large\cmark} &
\cellcolor{Gray} \multirow{-2}{*}{\green\large\cmark}   \\

\multicolumn{1}{c|}{}&{ TOS  } &
 & &
\\
\multicolumn{1}{c|}{\multirow{4}{*}}&
\multirow{-1}{*}{\scriptsize\citep{davis2017three}}&
\multirow{-2}{*}{\red\large\xmark}
&
 \multirow{-2}{*}{\green\large\cmark} &  \multirow{-2}{*}{\green\large\cmark} \\

\multicolumn{1}{c|}{\multirow{4}{*}}&\multirow{1}{*}{\cellcolor{Gray} LSAO-TOS  } &
\cellcolor{Gray} &\cellcolor{Gray}&\cellcolor{Gray} \\
\multicolumn{1}{c|}{\multirow{4}{*}}&
\multirow{-1}{*}{\scriptsize\citep{giselsson2016line}}\cellcolor{Gray}&
\multirow{-2}{*}{\green\large\cmark} \cellcolor{Gray}
\cellcolor{Gray}&
\cellcolor{Gray} \multirow{-2}{*}{ ~\green{\large\cmark}\footnote{\label{foot:conv_residuals}Convergence rate in terms of  operator residuals. }}&
\cellcolor{Gray} \multirow{-2}{*}{ \red{\large\xmark}  }\\

\multicolumn{1}{c|}{} &\multirow{-1}{*}{{PDHG}}  &\multirow{2}{*}{\red\large\xmark} & \multirow{2}{*}{ \green{\large\cmark}} & \multirow{2}{*}{\red{\large\xmark}}
  \\
\multicolumn{1}{c|}{\multirow{4}{*}}& \multirow{-1}{*}{\scriptsize\citep{condat2013primal, vu2013splitting}} &  &  & \\

\multicolumn{1}{c|}{}&{\cellcolor{Gray} PDHG-LS  } &
\cellcolor{Gray} &\cellcolor{Gray}&\cellcolor{Gray} \\
\multicolumn{1}{c|}{\multirow{4}{*}}&
\multirow{-1}{*}{\scriptsize\citep{malitsky2016first}}\cellcolor{Gray}&
\multirow{-2}{*}{\green\large\cmark} \cellcolor{Gray}
\cellcolor{Gray}&
\cellcolor{Gray} \multirow{-2}{*}{\red\large\xmark} &
\cellcolor{Gray} \multirow{-2}{*}{\red\large\xmark}  \\
\cline{2-5}\vspace{-2.8ex}\\
\end{tabular}
}

\section{Methods}\label{scs:methods}

In this section we present our main contribution, a three operator splitting method with adaptive step-size. The method is detailed in Algorithm~\ref{alg:algo_three_ls} and
 requires at each iteration to evaluate once the gradient of $f$ and the proximal operators of $g$ and $h$, and perform two function evaluations of $f$.
At iteration $t$ the candidate step-size $\gamma_t$ is chosen as to verify the following \emph{sufficient decrease} condition between the iterates $\zz_{t}$ and $\xx_{t+1}$ (Line~\ref{line:ls}):
\begin{gather*}
f(\xx_{t+1}) \leq Q_t(\xx_{t+1}, \gamma_t)\,,\,\text{ with \,$Q_t$ defined as}\\
Q_t(\xx, \gamma)\!\defas\!f(\zz_t) + \langle \nabla f(\zz_t), \xx\! -\! \zz_t \rangle\! +\! \frac{1}{2\gamma} \|\xx\!-\!\zz_t\|^2.
\end{gather*}
\begin{minipage}[t]{.48\linewidth}
This inequality can be interpreted as a quadratic upper bound condition on  $f$ at $\xx_{t+1}$: the right-hand side is a quadratic $Q_{t}$ which is tangent to $f$ at $\zz_t$ with amplitude $(2 \gamma_t)^{-1}$, and both sides are evaluated at $\xx_{t+1}$, defined in Line~\ref{line:x_update}. The under- \unskip\parfillskip 0pt \par
\end{minipage}
\begin{minipage}[t]{.3\linewidth}\vspace{-0.5em}
\hspace*{1em}{\begin{tikzpicture}[domain=-0.1:3, baseline=(current bounding box.north)]
  % \draw[very thin,color=gray] (-0.1,-0.1) grid (3.9,3.9);
  \draw[->] (-0.2,0) -- (2.8,0) node[below] {$\xx$};
  \draw[->] (0,-0.2) -- (0,2.9) node[above] {};
  \node[circle,fill=black,inner sep=0pt,minimum size=3pt,label=below:{$\zz_t$}] (a) at (0.5,0) {};
  \node[circle,fill=black,inner sep=0pt,minimum size=3pt,label=below:{$\xx_{t+1}$}] (a) at (1.7,0) {};
  \node[circle,draw=black, fill=white, inner sep=0pt,minimum size=5pt] (b) at (0.5 ,1.68) {};
  \draw[gray, dashed] (0.5,0)--(0.5,1.68);
  \draw[gray, dashed] (1.7,0)--(1.7,0.68);
  \draw[color=black!70,densely dashed, domain=0.0:2.5, line width=0.4mm] plot ({\x}, {1.7 - 1.48 * (\x - 0.5) + 0.8 * (\x-0.5) * (\x-0.5)}) node[above, align=left] {$Q_{t}(\xx)$};;
%   \draw[color=black!80, dashed, domain=0.4:1.99, line width=0.2mm] plot ({\x}, {1.7 - 1.48 * (\x - 0.5) }) node[right, align=left] {$\nabla f(\zz_t)$};;
  % \draw[color=blue] plot (\x, { sin(\x r) }) node[right] {$f(x) = \sin x$};
  \draw[color=black, line width=0.5mm] plot (\x, { 3 * exp(-\x-0.2) +0.2})
    node[above] {\vspace{0em}$f(\xx)$};
\end{tikzpicture}}
% \vspace{-1.5em}\caption{The line search condition \eqref{eq:line_search_condition} ensures that the quadratic upper bound $Q_f$ at $\zz_t$ remains an upper bound at $\xx_t$.}
\end{minipage}

\vspace{-0.4em} lying principle of choosing the step-size based on the minimization of a quadratic upper bound has already been successful for the proximal-gradient method, where it is also referred to as backtracking~\citep{beck2009gradient} or full relaxation~\citep{nesterov2013gradient}. 
In fact, the proposed method coincides with the aforementioned when one of the proximal terms is constant.

\vspace{0.5em}By the properties of $L_f$-smooth functions, the sufficient decrease condition is verified for any $\gamma_t \leq 1/L_f$. Hence the
step-size search loop always has a finite terminationand the step-size is lower bounded by $\gamma_t \geq \min\{\tau / L_f, \gamma_0\}$. 
The practical advantage of this strategy is that it allows to consider a step-size potentially larger than $1/L_f$ and verify whether the above is verified at each iteration. If it is, then the algorithm uses the current step-size, and if not, it decreases the step-size by a factor which we denote $\tau$.

\paragraph{Growing step-size strategies.} 
We consider two different strategies to initialize next iterate step-size. The first strategy (\emph{Variant 1}) is the simplest and consists in initializing the next step-size with the current one (Line~\ref{line:update_gamma_v1}). In this variant, the step-size is only allowed to decrease.

The second strategy (\emph{Variant 2}) allows the step-size to increase but in exchange requires the proximal term $h$ to be Lipschitz continuous (note, not smooth as $f$ but only Lipschitz).
This is the case of most penalties (i.e., $\ell_1$, group lasso, total variation, etc.) but not of indicator functions and so
is less general than the first variant. As we will see in the applications section, the ability to grow the step-size has an important effect on its empirical performance.
% {\blue other adaptive variants don't seem to require Lipschitz continuity?}

\setlength{\textfloatsep}{20pt}% Less vertical space after algorithm
\begin{algorithm}[t]
 \KwIn{$\zz_0 \in \RR^p$, $\uu_0 \in \RR^p$, $\gamma_0 > 0, \tau \in (0, 1)$}

\For{$t=0, 1, 2, \ldots$ }{

\Repeat(\hfill $\triangleright$ step-size search loop){

$\xx_{t+1} = \prox_{\gamma_{t} g}(\zz_{t} - \gamma_t\uu_{t} - \gamma_t\nabla f(\zz_{t}))$\label{line:x_update}

% $q_t =\!f(\zz_{t-1}) + \langle \rr_t,\xx_t\!-\!\zz_{t-1} \rangle+ \mfrac{1}{2\gamma_{t}} \|\xx_t\!-\!\zz_{t-1}\|^2$

\eIf( ){$f(\xx_{t+1}) \leq Q_t(\xx_{t+1}, \gamma_t)~$\label{line:ls}}{

{\bfseries break}\hfill$\triangleright$ sufficient decrease verified
}{
$\gamma_t = \tau \gamma_t$\hfill $\triangleright$ decrease step-size\label{alg:decrease_step_size}
}
}{
}
$\zz_{t+1} = \prox_{\gamma_t h}(\xx_{t+1} + \gamma_t \uu_{t})$\label{line:z_update}

$\uu_{t+1} = \uu_{t} + (\xx_{t+1} - \zz_{t+1}) / \gamma_t$\label{line:u_update}

\hfill$\triangleright$ choose step-size for next iteration, two variants

% {\bfseries Variant 1}
\Variantone{

$\gamma_{t+1} = \gamma_{t}$\label{line:update_gamma_v1}
}

% {\bfseries Variant 2} ()
\Varianttwo(\hfill $\triangleright$ only if $h$ is $\beta_h$-Lipschitz){

$\delta_t = Q_t(\xx_{t+1}, \gamma_t) - f(\xx_{t+1})$ \label{line:define_delta}

Choose any $\gamma_{t+1} \in [\gamma_t, \sqrt{\gamma_{t}^2 +  \gamma_t \delta_t (2\beta_h)^{-2}}]$\label{line:increase_ls}
}

}
\Return{$\xx_{t+1}$, $\uu_{t+1}$}
 \caption{Adaptive Three Operator Splitting}\label{alg:algo_three_ls}
\end{algorithm}

\paragraph{Initial and default values.} The proposed method takes as input 4 parameters, which we briefly discuss, together with a growing step-size heuristic for Variant 2:
\begin{itemize}[leftmargin=*]
\item \emph{Initial guess $\zz_0$ and $\uu_0$}. $\zz_0$ is an initial guess of the primal problem \eqref{eq:obj_fun}, while $\uu_0$ is an initial guess for a minimizer of a (yet to be defined) dual function \eqref{eq:dual_loss}. In practice, we initialize both variables to zero.
\item \emph{Initial step-size $\gamma_0$}. To estimate a starting value for the step-size, we start with $\varepsilon=10^{-3}$, $\widetilde\zz = \zz_0 - \varepsilon \nabla f(\zz_0)$ and divide $\varepsilon$ by 10 until $f(\widetilde\zz) \leq f(\zz_0)$. Then we solve $f(\widetilde\zz) = Q_0(\widetilde\zz)$  for $\gamma_0$ and double that estimate, giving
\begin{equation}
  \gamma_0 = 4 (f(\zz_0) - f(\widetilde\zz)) \|\nabla f(\zz_0)\|^{-2}~.
\end{equation}
\item The \emph{line search decrease parameter $\tau$} regulates the factor by which the step-size is decreased each time the line search condition is unsuccessful.
This is a parameter that is common to all line search methods and can be set to any value $\tau \in (0, 1)$. Following~\citep{malitsky2016first} we set it to $\tau =0.7$.
\item \emph{step-size growth.} Variant 2 allows the step-size to grow by an amount that depends on $\beta_h^{-2}$. This quantity can be arbitrarily large (e.g., vanishing regularization), and so choosing the largest admissible step-size might result in too many decrease corrections. This can be avoided e.g. by limiting its growth to double every $20$ iterations. Line \ref{line:increase_ls} then becomes:
\begin{equation}
  \gamma_{t+1} = \min\{\gamma_t 2^{0.05}, \sqrt{\gamma_{t}^2 +  \gamma_t \delta_t (2\beta_h)^{-2}}\}~.
\end{equation}
\end{itemize}

%
% One could argue that the method is not fully hyperparameter-free as it requires a choice of $\tau$. This parameter is common to all line search strategies and has a very marginal influence on the method's performance. In our case this parameter was never tuned and set to $\tau=0.5$ for all experiments.

Upon termination, the algorithm returns two vectors. The first vector is an approximate solution to \eqref{eq:obj_fun}, while the second vector is an approximate minimizer of a dual objective which we will detail in \S\ref{scs:analysis}.

\paragraph{Special cases and related methods.} We mention two notable special cases of this algorithm. First, for any step-size $\gamma_t \leq 1/L_f   $, the line search condition will always succeed by the properties of $L_f$-smooth functions and so the step-size in Variant 1 is constant. Defining $\yy_t = \xx_t + \gamma_t \uu_{t-1}$, it is easy to verify that Algorithm~\ref{alg:algo_three_ls} (Variant 1) can be written with a constant step-size $\gamma = \gamma_t$ as an iteration of the form
\begin{equation}\begin{aligned}
  &\zz_t = \prox_{\gamma h}(\yy_t)\\
  &\xx_{t+1}\!=\!\prox_{\gamma g}(2 \zz_t - \yy_t - \gamma \nabla f(\zz_t))\\
  &\yy_{t+1} = \yy_t - \zz_t + \xx_{t+1}~,
  \end{aligned}\end{equation}
which is the standard (non-overrelaxed) form of the three operator splitting~\citep[Algorithm 1]{davis2017three}. 
The adaptive variant requires two extra function evaluations $f(\zz_{t})$ and $f(\xx_{t+1})$ for the line search condition in Line~\ref{line:ls}, but as we will see in the experimental section, most often the ability to take larger step outweighs this extra cost.

Second, for $h=0$, we have from lines \ref{line:z_update} and \ref{line:u_update} that $\uu_t = 0$ and in this case (ignoring growing step-size strategies), this algorithm simplifies to the proximal gradient descent with line search of \citep{beck2009gradient}.

Algorithm~\ref{alg:algo_three_ls} can be written equivalently in a way that highlights similarities and differences with the PDHG method. Using Moreau's decomposition $\prox_{\gamma h}(\xx) = \xx - \gamma \prox_{\gamma h^*}(\xx/\gamma)$ 
 yields the following recurrence
\begin{align}
\uu_{t+1} &= \prox_{h^*/\gamma}(\uu_{t} +\xx_{t}/\gamma)~,\\
\xx_{t+2} &= \prox_{\gamma g}(\xx_{t+1}\!- \gamma (\nabla f(\zz_{t+1})+ 2 \uu_{t+1}\!- \uu_{t}))\,.\nonumber
\end{align}
This form is almost identical to Algorithm 3.2 in~\citep{condat2013primal}, but with a different step-size and the gradient evaluated at the extrapolated $\zz_{t+1} = \xx_{t+1} - \gamma (\uu_{t+1} - \uu_{t})$ instead of the previous iterate $\xx_{t+1}$ in PDHG.

\subsection{Extension to multiple proximal terms}\label{scs:extension_k_terms}

We now consider the problem of minimizing an objective of the form:
\begin{empheq}[box=\mybluebox]{equation}\tag{OPT-$k$}\label{eq:obj_fun_k}
  \minimize_{\xx \in \RR^p}\,\vphantom{\sum_i^n} \varphi(\xx) + \textstyle\sum_{j=1}^k h_j(\xx) ~,
\end{empheq}
where $\varphi$ is $L_\varphi$-smooth and each $h_j$ is proximal.
The adaptive three operator splitting can be used to solve problems of this form by reducing them to a problem of the form~\eqref{eq:obj_fun} in an enlarged space.
Consider consider the following problem in $\RR^{k\times p}$,
\begin{equation*}
\minimize_{\XX \in \RR^{k\times p}}\underbrace{\varphi(\overline\XX)}_{= f(\XX)} + \underbrace{\textstyle\sum_{j=1}^k h_j(\XX_j)}_{= h(\XX)} + \underbrace{\imath\{\XX_1\!=\!\cdots\!=\!\XX_k\}}_{=g(\XX)}.
\end{equation*}
It is easy to see that this problem
shares the same set of solutions as \eqref{eq:obj_fun_k} with the correspondence $\xx = \overline{\XX}$, as the last term forces all the $\XX_i$ terms to be equal.
In this formulation the first term is smooth, the second term is proximal (variables in $h_i$ are separated) and the proximal operator of the last term is given by $\overline{\XX}\,\mathbf{1}^T$. Hence Algorithm~\ref{alg:algo_three_ls} can be applied to this problem.
Deriving the complete algorithm is now merely a matter of replacing $f, g, h$ by its appropriate values in Algorithm~\ref{alg:algo_three_ls} and is specified in \ref{apx:k_operator_splitting}.
The resulting adaptive algorithm
seems to be new also in this extended formulation.

It is also possible to swap the definitions of $g$ and $h$, which results in a different algorithm that can be seen as an adaptive variant of the
 Generalized Forward-Backward splitting of \citet{raguet2013generalized}. However, this formulation is less convenient for our purpose, since in this case the $h$ term is always an indicator function and so it would not be possible to apply variant 2 of our algorithm.

\section{Analysis}\label{scs:analysis}

In this section we provide a convergence rate analysis of the proposed method. We start by a characterization the set of fixed points of the algorithm, followed by a discussion on the gap function used to measure suboptimality. Finally, we present convergence rates for two different function classes.
All proofs can be found in \ref{apx:analysis}.

{\bfseries Assumption 1: Regularity.} We assume that $f$ is convex and $L_f$-smooth in $\RR^p$ and that $g$ and $h$ are proper (i.e., have nonempty domain), lower semicontinuous (i.e., its sublevel sets are closed) convex functions. We note that lower semicontinuity is a weak form of continuity that allows extended-valued functions (such as the indicator function) over a closed domain.

{\bfseries Assumption 2: Qualification conditions.} We assume the relative interior of $\dom g$ and $\dom h$ have a non-empty intersection. This is a weak and standard assumption that, together with the regularity assumption, guarantees strong (also known as total) duality~\citep[Prop. 5.3.8]{bertsekas2015convex}.
% This is a standard assumption that allows to associate a \emph{dual} problem with the {\blue same optimal objective than our original objective}.

Using the definition of Fenchel conjugate, we can can reformulate \eqref{eq:obj_fun} as a saddle-point problem:
\begin{align}
    &\min_{\xx \in \RR^d} f(\xx) + g(\xx) + h(\xx) \\
    &\quad = \min_{\xx \in \RR^d}  f(\xx) + g(\xx) + \max_{\uu \in \RR^d}\left\{\langle \xx, \uu \rangle - h^*(\uu)\right\}\\
    &\quad= \min_{\xx \in \RR^d} \max_{\uu \in \RR^d} \underbrace{f(\xx) + g(\xx) + \langle \xx, \uu \rangle - h^*(\uu)}_{\defas \mathcal{L}(\xx, \uu)}~.
\end{align}
We recall that a saddle point of $\mathcal{L}$ is a pair $(\xx^\star, \uu^\star)$ such that the following is verified for any $(\xx, \balpha)$ in the domain~\citep[\S4.1]{hiriart2013convex}:
\begin{equation}\label{eq:saddle_point}
   \mathcal{L}(\xx^\star\!, \balpha) \leq \mathcal{L}(\xx, \uu^\star) ~.
\end{equation}
A consequence of strong duality is the equivalence between the saddle points of $\mathcal{L}$ and the minimizers of the primal and dual objectives. Let $P$ and $D$ denote these primal and dual objectives:
\begin{align}
    P(\xx) &\defas f(\xx) + g(\xx) + h(\xx)\label{eq:primal_loss}\\
    D(\uu) &\defas (f+g)^*(-\uu) + h^*(\uu).\label{eq:dual_loss}
\end{align}
Then if $(\xx^\star, \uu^\star$) is a saddle point of $\mathcal{L}$, $\xx^\star$ is a minimizer of $P$ and $\uu^\star$ is a minimizer of $D$. Likewise, a pair of minimizers of $P$ and $D$ form a saddle point of $\mathcal{L}$.

\subsection{Fixed point characterization}\label{scs:disguise}

A common first step in the analysis of optimization methods is the study of its set of fixed or stationary points. While this does not necessarily imply convergence, knowing which elements will be left invariant by the method improves our understanding and is a stepping stone for further analysis.
We will show that the set of fixed points of the algorithm has a particularly simple and elegant structure: the Cartesian product of primal and dual solutions.

For the purpose of analysis it will be useful to express Algorithm~\ref{alg:algo_three_ls} as an iteration of the form, $(\zz_{t+1}, \uu_{t+1}) = \boldsymbol{T}_{\gamma_t}(\zz_{t}, \uu_{t})$, where the operator $\boldsymbol{T}_{\gamma}$ is defined as
\begin{align}
\boldsymbol{T}_{\gamma}&(\zz,\uu) \defas (\zz^+,\uu^+), \text{ with }\\
&\begin{cases}
\zz^+ = \prox_{\gamma h}(\xx(\zz, \uu) + \gamma \uu)\\
\uu^+ = \uu + (\xx(\zz, \uu) - \zz^+) / \gamma\\
\xx(\zz, \uu) = \prox_{\gamma g}(\zz - \gamma(\uu + \nabla f(\zz)))~.
\end{cases}\nonumber
\end{align}

The following theorem characterizes the set of fixed points of this operator, denoted $\Fix(\boldsymbol{T}_{\gamma})$.
\begin{theorem}\label{thm:fixed_point}
Let $\mathcal{P}^\star$ denote the set of minimizers of the primal objective $P$ and $\mathcal{D}^\star$ the set of minimizers of the dual objective $D$\,. Then the fixed points of $\TT_\gamma$ are given by
\begin{equation}\label{eq:fix_decomposition}
\Fix(\boldsymbol{T}_{\gamma}) = \mathcal{P}^\star \times \mathcal{D}^\star~.
\end{equation}
\end{theorem}

\subsection{Gap function}

The progress of optimization methods is commonly measured in terms of a gap or merit function that is zero at optimum and nonzero otherwise.
An appropriate gap function for many first-order methods is the suboptimality of the objective function.
However, this is not an appropriate suboptimality measure for this algorithm, as the objective function might be $+\infty$ at an iterate, for example when the two proximal terms are an indicator function.

\citet{davis2015three} avoid the issue by either evaluating $h$ and $g$ at different iterates \citep[Corollary D.5.1]{davis2015three} or assuming Lipschitz continuity of one of the proximal terms \citep[Corollary D.5.2]{davis2015three}.

In this work we take an alternative approach, and instead use the following \emph{saddle point suboptimality} criterion to measure the progress of our algorithm:
\begin{equation}\label{eq:gap_function}
  {\mathcal{L}(\xx_{t+1}, \uu) - \mathcal{L}(\xx, \uu_{t+1})}~.
\end{equation}
From the definition of saddle point in Eq.~\eqref{eq:saddle_point},
this criterion is non-positive for all $(\xx, \uu)$ if and only if $(\xx_{t+1}, \uu_{t+1})$ is a saddle point, and is so an appropriate suboptimality criterion.
Furthermore, contrary to the primal objective function, this is defined for all iterates without further assumptions.
Finally, we mention that this criteria has been previously used in the analysis of primal-dual methods, see e.g., \citet{chambolle2016introduction,chambolle2016ergodic} and \citet{gidel17a} for a discussion of saddle point gap functions.

This suboptimality criteria can also be related to the primal and dual gap, as minimizing \eqref{eq:gap_function} over $\xx$ and maximizing over $\uu$ one recovers the primal-dual gap $P(\xx_t) - D(\uu_t)$ by definition of Fenchel conjugate.

\subsection{Sublinear convergence}

The following theorem gives a sublinear convergence rate for Algorithm~\ref{alg:algo_three_ls}. This convergence will be given in terms of the weighted ergodic (i.e., averaged) sequence.
Denoting by $s_t$ the sum of all step-sizes up to iteration $t$, i.e., $s_t \defas \sum_{i=0}^{t-1} \gamma_t$, the ergodic iterates $\overline{\xx}_t$ and $\overline{\uu}_t$ are defined as
\begin{equation}
\overline{\xx}_t \defas \Big(\sum_{i=0}^{t-1}\gamma_i \xx_{i+1}\Big)/{s_t } \,,\;\; \overline{\uu}_t \defas \Big({\sum_{i=0}^{t-1}\gamma_i \uu_{i+1}}\Big)/{s_t }\,.
\end{equation}
While results in this subsection will be stated in terms of this ergodic sequence, in practice the last iterate gives most often a better empirical convergence, see e.g.,~\citep[\S7.2.1]{chambolle2016ergodic} for a discussion of this phenomenon. For a more theoretically-sound algorithm, one can compare the objective at the ergodic and last iterate, and return the one with smallest objective.

\begin{theorem}[sublinear convergence rate]\label{thm:sublinear}
For every ${t \geq 0}$ and any $(\xx, \balpha)$ in the domain of $\mathcal{L}$ we have the following convergence rate for Algorithm~\ref{alg:algo_three_ls} (both variants):
\begin{equation*}\label{eq:ergodic_convergence_rate}
\mathcal{L}(\overline\xx_{t+1}, \uu)- \mathcal{L}(\xx, \overline\uu_{t+1}) \leq \frac{\|\zz_0 - \xx\|^2 + {\gamma_0^2\|\uu_0 - \uu\|^2}}{2 s_t}\,.
\end{equation*}
\end{theorem}
\paragraph{Convergence in terms of function value suboptimality.}
The previous result gives an $\mathcal{O}(1/t)$ convergence rate for arbitrary convex functions in terms of the saddle point suboptimality. As we have discussed previously, it is not possible to obtain similar rates in terms of the function suboptimality without further assumptions.
We will now show that it is sufficient to assume Lipschitz continuity on $h$ to derive from the previous theorem a convergence rate in terms of the primal function suboptimality.

The following Corollary can be obtained by optimizing with respect to $\uu$ the bound in the previous theorem and using the Lipschitz continuity to bound $\|\uu_0 - \uu\|^2$. This gives an $\mathcal{O}(1/t)$ convergence rate for the primal function suboptimality, roughly matching that of \citet[Corollary D.5.2]{davis2015three} for the non adaptive  variant:
\begin{corollary}\label{cor:sublinear_convergence}
Let $h$ be $\beta_h$-Lipschitz. Then, we have the following rate for the weighted ergodic iterate on the objective $P(\xx) \defas f(\xx) + g(\xx) + h(\xx)$:
\begin{equation*}
P(\overline\xx_{t+1}) - P(\xx^\star) \leq  \frac{\|\zz_0 - \xx^\star\|^2\!+ 2 \gamma_0^2(\|\uu_0\|^2\!+ \beta_h^2)}{2 s_t}\,.
\end{equation*}
\end{corollary}

\subsection{Linear convergence}

In this subsection we assume that $f$ is $\mu_f$-strongly convex and $h$ is $L_h$-smooth (with $0 < \mu_f, 0 < L_h < +\infty$). We denote by $\xx^\star$ the minimizer of the primal loss (unique by strong convexity of $P$) and by $\uu^\star$ the minimizer of the dual loss (also unique by strong convexity of $D$, consequence of the $L_h$-smoothness of $h$).

The convergence rates will be given in terms of the following quantities
\begin{equation}\label{eq:def_linear_rates}
\begin{aligned}
   &\rho \defas \mu_f\min\{\gamma_0,  \tau / L_f\}  ~,~  \sigma \defas 1/(1 + \gamma_0 L_h)\\
   &\qquad\qquad\quad\xi \defas {\mu_f}/{(\mu_f + L_h)}~.
\end{aligned}
\end{equation}
All these belong to the interval $(0, 1)$. Assuming $\gamma_0 \geq \tau/L_f$, then
$\rho$ is the inverse of $f$'s condition number, while $\sigma$ depends the smoothness of $h$. $\xi$ is only used by variant 2 and is a less tight bound that $\sigma$ that depends on both the strong convexity of $f$ and smoothness of $h$. Note that by strong convexity, $\gamma_0 < 1/\mu_f$ as otherwise the sufficient decrease condition would not succeed and so $\sigma \geq \xi$.

\begin{theorem}[linear convergence rate]\label{thm:linear_convergence}
Let $\xx_{t+1}, \uu_{t+1}$ be the iterates  produced by Algorithm~\ref{alg:algo_three_ls} after $t$ iterations. Then we have the following linear convergence for Variant 1 (V1) and Variant 2 (V2):
\begin{align}
  &\text{V1}: \|\xx_{t+1} - \xx^\star\|^2 \leq \Big(1 - \min\big\{\rho, \sigma\big\}\Big)^{t+1} D_0\\
  &\text{V2}: \|\xx_{t+1} - \xx^\star\|^2 \leq \Big(1 - \min\big\{\rho,\xi, \mfrac{1}{2}\big\}\Big)^{t+1} E_0~,
\end{align}
with $D_0 \defas 6\|\zz_0 - \xx^\star\|^2 + \frac{6}{1 - \sigma}\|\gamma_0(\uu_0 - \uu^\star)\|^2$
 and ${E_0 \defas 6\|\zz_0 - \xx^\star\|^2 + \frac{6}{1 - \xi}\|\gamma_0(\uu_0 - \uu^\star)\|^2}$.
\end{theorem}

\paragraph{Discussion.} For $\gamma_t = {1}/{L_f}$, the sufficient decrease condition is always verified and the algorithm can be run with $\tau=1$. In this case, Variant 1 of Algorithm~\ref{alg:algo_three_ls} defaults to TOS, and we can compare the obtained rates with those in \citep{davis2015three}.

While the sublinear convergence rate obtained in Corollary~\ref{cor:sublinear_convergence} roughly matches the rate obtained in (\citet[Corollary D.5.2, see our \ref{apx:comparison_convergence_rates}]{davis2015three}),
the linear convergence rates are instead significantly different. The linear convergence rate obtained in  \citep[Theorem D.6.6]{davis2015three}, after optimizing for all parameters, yields a rate of $\rho \sigma^{2}$, which is \emph{strictly worse} than the $\min\{\rho, \sigma\}$ rate that we obtained. This difference can be quite large, e.g., for $\rho=\sigma$ this becomes $\rho$ versus $\rho^{3}$.

Finally, we note that the number of evaluations of the sufficient decrease condition can be bounded as in \citep[Lemma 4]{nesterov2013gradient}.

\section{Applications} \label{scs:applications}

\subsection{Learning with Multiple Penalties}\label{scs:learning_multiple_penalties}

In this subsection we discuss how some penalties with costly to compute proximal operator  can be decomposed as a sum of proximal terms and so fall within the current framework. The exact expression of the proximal operators is given in~\ref{apx:proximal_operators}.

\paragraph{Group lasso with overlap.} \citet{jacob2009group} generalized group $\ell_1$ norm by allowing each variable to belong to more than one group, thereby introducing overlaps among groups and allowing for more complex prior knowledge on the structure. For a set of subindices $\mathcal{G}$ which we will call groups, this penalty is defined as $\|\xx\|_\mathcal{G} = \sum_{G \in \mathcal{G}}\|[\xx]_G\|_2$. If each coefficient is at most in $s$ groups, then $\mathcal{G}$ can be decomposed as $\mathcal{G}= \mathcal{G}_1 \cup \ldots \cup \mathcal{G}_s$, where the $\mathcal{G}_i$ are disjoint. This allows to express the group lasso with overlap as a sum of $s$ non-overlapping group lasso penalties, for which the proximal operator has a closed form expression.

\paragraph{Multidimensional total variation.} For the task of image restoration and denoising it is common to consider a regularization term in the form of a total variation regularizer. For an image $\xx$ of size $p \times q$, the 2-dimensional total variation norm $\|{\boldsymbol X}\|_{\text{TV}}$ is defined as
\begin{equation*}
 \underbrace{{\sum_{i=1}^p} {\sum_{j=1}^{q-1}} |{\boldsymbol X}_{i, j+1} - {\boldsymbol X}_{i, j}|}_{ = g(\XX)} + \underbrace{\sum_{j=1}^q {\sum_{ji=1}^{p-1}} |{\boldsymbol X}_{i+1, j} - {\boldsymbol X}_{i, j}|}_{ = h(\XX)}\,.
 \end{equation*}
From here we recognize that $g$ and $h$ are fused lasso (also known as 1D-total variation) penalties acting on the columns and rows of $\XX$ respectively.
Efficient methods to evaluate the proximal operator of the fused lasso penalty have been developed by~\citet{condat2013direct, johnson2013dynamic}.

\paragraph{Isotonic and nearly isotonic penalties.} In some applications there exists a natural ordering between variables: $\xx_1 \leq \xx_2 \leq \cdots \leq \xx_p$.
This can be enforced through constraints, and the projection onto these is known as isotonic regression~\citep{Best1990}.
The indicator function over the set of constraints can also be split into a sum of two proximal terms (see \ref{apx:isotonic}) as
\begin{align}
  &\imath\{\xx_1 \leq \xx_2 \leq \xx_3 \leq \xx_4 \leq  \cdots\}\\
  &= \underbrace{\imath\{\xx_1\! \leq \!\xx_2; \xx_3 \leq \xx_4;  \cdots\}}_{=g(\xx)}  + \underbrace{\imath\{\xx_2 \leq \xx_3; \xx_4 \leq \xx_5; \cdots\}}_{=h(\xx)}\,. \nonumber
\end{align}
In cases in which the variables are only ``mostly'' non-decreasing, the constraint can be relaxed via a nearly-isotonic penalty~\citep{tibshirani2011nearly} of the form $\sum_{i=1}^{p-1}\max\{\xx_i - \xx_{i+1}, 0\}$, in which only the non-increasing coefficients are penalized. This penalty can be split  the same way as the isotonic constraints above.

\paragraph{$\ell_1$ trend filtering.}  This penalty is defined by the absolute value of the second order differences and promotes piecewise-linear coefficients \citep{kim2009ell1}. It is defined as
$\|\xx\|_\text{TF} \defas \textstyle\sum_{i=1}^{p-2} |\xx_i - 2 \xx_{i+1} + \xx_{i+2}|$.
We can split this sum into 3 proximal terms such that the resulting terms: the $j$-th term contains the factors for which $i$ is congruent to 3 modulo $j$.

\paragraph{Constraints over doubly stochastic matrices.} Optimization problems with constraints on the set of doubly stochastic matrices appear in many convex relaxations of combinatorial problems such as seriation~\citep{fogel2013convex}, quadratic assignment~\citep{lawler1963quadratic} and graph matching~\citep{conte2004thirty,Aflalo2015}. The set of double stochastic matrices is composed of square matrices with nonnegative entries, each of whose rows and columns sum to $1$, i.e., $\{\XX^T \boldsymbol{1} = \boldsymbol{1}, \XX \boldsymbol{1} = \boldsymbol{1}, \XX \geq \boldsymbol{0}\}$. The indicator function over this set can be split as
\begin{equation}
\underbrace{\imath\{\XX^T \boldsymbol{1} = \boldsymbol{1}, \XX \boldsymbol{1} = \boldsymbol{1}\}}_{=g(\XX)} + \underbrace{\imath\{\XX \geq \boldsymbol{0}\}}_{=h(\XX)} \quad,
\end{equation}
and the projection onto both sets is available in closed form~\citep[\S4.3]{lu2016fast}.

\begin{figure*}[tb!]
\includegraphics[width=\linewidth]{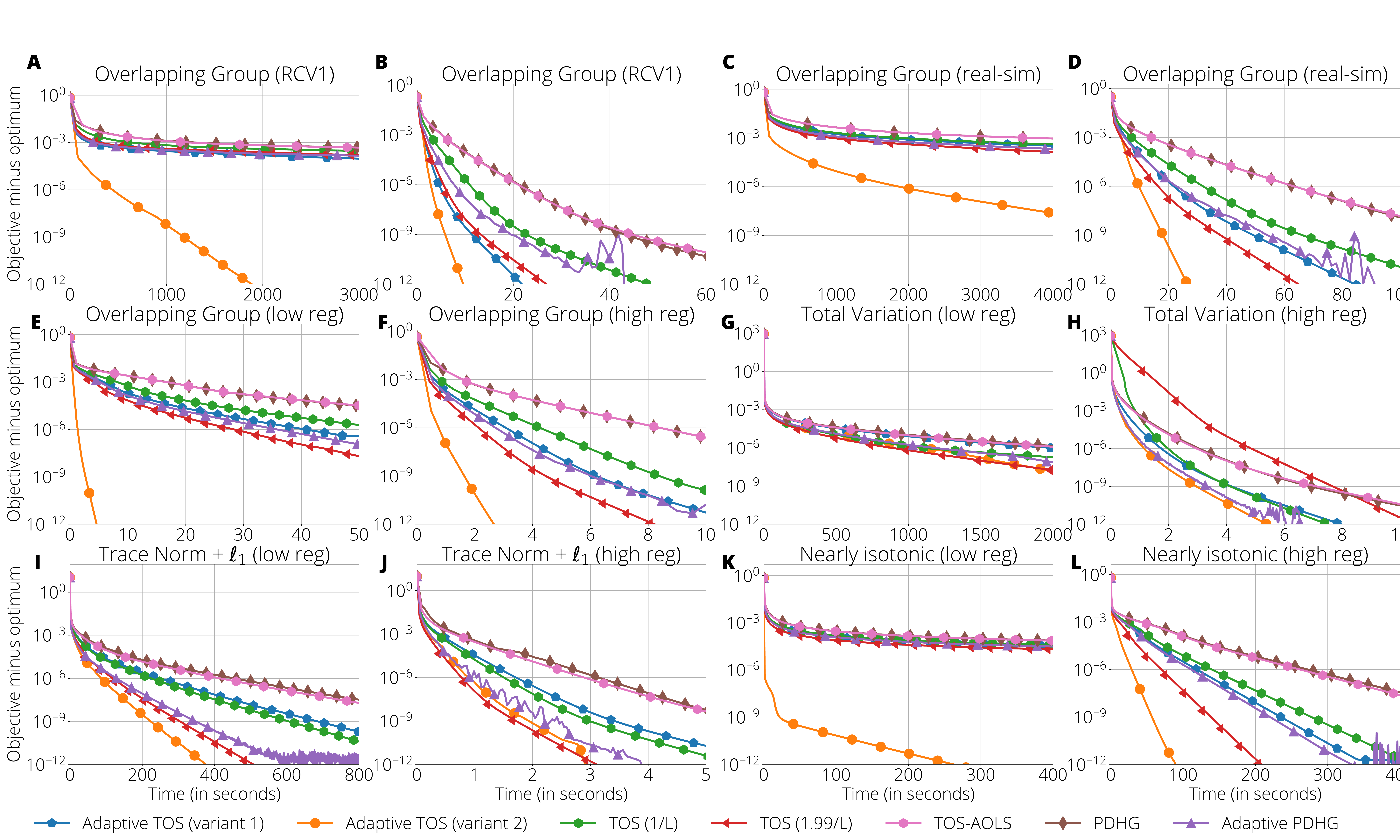}
\caption{{\bfseries Comparison of different proximal splitting methods}. The \emph{top row} gives result for two real datasets with an overlapping group lasso penalty. The \emph{second and third row} show results on synthetic datasets for 4 different penalties: overlapping group lasso (E, F), 2-dimensional total variation (G, H), trace norm + $\ell_1$ (I, J) and nearly isotonic (K, L). The Adaptive TOS (Variant 2, i.e., with growing step-sizes) is the best performing method on 10 out of 12 experiments, and roughly equivalent to the best performing method in the other 2 cases (G, J).
% At best it is over an order of magnitude faster than the second best one (Subfig. A and B) and at worse it is roughly equivalent to the best method (Subfig. F).
}\vspace{-0.5em}
\label{fig:main_figure}
\end{figure*}

\paragraph{Dispersive sparsity.} In some applications it is desirable to encourage dispersion of the sparse coefficients. This happens for example in the modeling of neural spiking, as the spikes are assumed to be spaced across time~\citep{hegde2009compressive}. \citet{el2015totally} showed that this behavior can be promoted by considering a penalty of the form $\|\xx\|_1 + \imath\{\boldsymbol B|\xx| \leq c\}$ for a matrix $\boldsymbol B$ and some predefined constant $c$, where $|\xx|$ denotes the component-wise absolute value. This penalty can be split into three proximal terms by the introduction of a dummy variable $\zz$, resulting in
$
\|\xx\|_1 + \imath\{\boldsymbol B \zz \leq \boldsymbol c\} + \imath\{\zz = |\xx|\}
$.

{\bfseries Combination by addition}. A popular method to promote the joint behavior of different penalties is by adding them.
% which enters into this framework when the penalties are proximal (or sum or proximal terms).
This has been used to successfully learn models with sparse and nonnegative coefficients~\citep{yuan2007non}, sparse and low rank matrices~\citep{richard2012estimation}, sparse and piecewise constant~\citep{gramfort2013identifying}, to name a few.

\subsection{Benchmarks}\label{scs:benchmark}

In this subsection we provide an empirical evaluation of the proposed method. Due to space constraints we only give here a high level overview, deferring details as well as an extended set of experiments to \ref{apx:benchmarks}.
We consider the following methods:
\begin{itemize}[leftmargin=*]
  \item The proposed Adaptive TOS method (Algorithm~\ref{alg:algo_three_ls}), in its both variants.

  \item The TOS method of \citet{davis2015three}, with step-sizes $1/L_f$ and $1.99/L_f$ (the method is convergent for step-sizes $< 2/L_f$).

  \item The PDHG or Condat-V\~u algorithm~\citep{condat2013primal, vu2013splitting}, with step-sizes $\tau$ and $\beta/\tau$, where $\beta$ was chosen as the one giving the best overall performance over the grid $\beta=0.9, 0.5, 0.1$ (giving it a slight advantage).

  \item The adaptive PDHG of \citet{malitsky2016first}, with step-size hyperparameter $\beta$ chosen by the same technique as for PDHG.

  \item The averaged operator line search method of \citet{giselsson2016line} combined with TOS, named TOS-AOLS.
\end{itemize}
% All methods have been implemented in Python and except TOS-AOLS are part of the C-OPT optimization package~\citep{fabian_pedregosa_2017_437991}.

We compared these methods on 4 different problems and show the results in Figure~\ref{fig:main_figure}. In the first row we show the benchmarks on a logistic regression problem with overlapping group lasso penalty that we apply to two text datasets (\texttt{RCV1} and \texttt{real-sim}). Subfigures A and C were run with the regularization parameter chosen to give $50\%$ of sparsity, while B, E are run with higher levels of sparsity, chosen to give $5\%$ of sparsity.

In the second and third row we considered a battery of inverse problems with different penalties on synthetic datasets. These consists of a least squares (G, H, I, J) or logistic regression (rest) smooth term and 4 different penalties specified in the title of each plot (overlapping group lasso, total variation, trace norm $\ell_1$ and nearly isotonic, see \ref{apx:benchmarks} for a precise formulation). For each problem, we show 2 different benchmarks, corresponding to the low and high regularization regimes (denoted low reg and high reg).
We comment on a few trends from Fig.~\ref{fig:main_figure}:
\vspace{-0.5em}\begin{itemize}[leftmargin=*]
  \item {\bfseries Best performing method}.
  On 10 out of 12 experiments, the adaptive TOS algorithm (Variant 2) is the best performing  method, and in the other cases (E, H) its performance is roughly the same as that of the best performing method. In contrast, on 3 instances (A, I, K) it is an order of magnitude faster than the next method.
  \item {\bfseries Low vs high regularization regime.} The advantage of the adaptive method is highly correlated with the amount of regularization: in the low regularization regime, on 3 out of 6 the adaptive TOS is an order of magnitude faster then the fixed step-size method, while in the high regularization regime the difference shrinks and in the same problems is never more than a factor 2.6. %Note also that there is in general a large computing time difference between the low and high regularization regime.
  \item {\bfseries Uniform curvature}. The problems (G, H, I, J) in Fig.~\ref{fig:main_figure} use as smooth term a quadratic loss (i.e., constant Hessian), while the other methods use a logistic loss (non-constant Hessian). This suggests that the use of the adaptive step-size strategy (and in particular Variant 2 with its growing step-size) is more beneficial for smooth terms with non-uniform curvature. %, where $L_f$ is necessarily a coarse upper bound on its local value.
\end{itemize}

\section{Conclusion and Future Work}

We have presented and analyzed an adaptive step-size method to solve optimization problems consisting in a sum of a smooth term accessed through its gradient and two or more potentially non-smooth terms accessed through their proximal operator.
The method does not rely on any step-size hyperparameter (except for an initial estimate) and extensive empirical evaluation has showed computational gains on a variety of problems.
We mention two possible extensions of this work.

First, existing convergence results fail to fully explain their surprisingly good empirical convergence.
To the best of our knowledge, no work so far has derived linear convergence rates in absence of strong convexity and smoothness of one of the proximal terms for these methods (as is however empirically observed, see e.g. Figure~\ref{fig:main_figure}).

Second, it is an open question whether this or other adaptive step-size methods can be accelerated, as is the case of proximal gradient descent, which admits the adaptive FISTA variant~\citep{beck2009gradient}.

\clearpage

\ifarxiv
\section*{Acknowledgements}

FP would like to thank Wotao Yin, Walaa Moursi and Elvis Dohmatob for fruitful discussions, Michael Eickenberg for feedback on the manuscript and Patrick R. Johnstone for reporting a typo in an earlier version.

FP is funded through the European Union's Horizon 2020 research and innovation programme under the Marie Sklodorowska-Curie grant agreement 748900.
Computing time on was donated by Amazon through the program ``AWS Cloud Credits for Research''.

GG is funded by the Canada Excellence Research Chair in “Data Science for Realtime Decision-making” and by the NSERC Discovery Grant RGPIN-2017-06936.
\else
\fi

\bibliographystyle{icml2018}
\bibliography{biblio}

\clearpage
\appendix
\onecolumn
\titleformat{\section}{\Large\bfseries}{\thesection}{1em}{}
\titleformat{\subsection}{\large\bfseries}{\thesubsection}{1em}{}
\gdef\thesection{Appendix \Alph{section}}

\fontsize{11pt}{12pt}\selectfont

{\centering{\LARGE\bfseries Adaptive Three Operator Splitting}

\vspace{1em}
\centering{{\LARGE\bfseries Supplementary material}}

}
\vspace{2em}

\section{Basic definitions and properties}\label{apx:basic_definitions}

\vskip 1em

\begin{definition}[proper function] A function $f: \mathcal{X} \subseteq \RR^p \to ]-\infty, \infty]$ is said to be proper if its domain not empty.
\end{definition}

\vskip 0.5em

\begin{definition}[Fenchel conjugate] The Fenchel conjugate of a function $f: \mathcal{X} \subseteq \RR^p \to ]-\infty, \infty]$ is defined as
\begin{equation}
f^*(\xx^\star) = \sup_{ \xx \in \mathcal{X}} \, \langle \xx^{\star} , \xx \rangle - f ( \xx ) ~.
\end{equation}
\end{definition}

\vskip 0.5em

\begin{definition}[lower semicontinuity] We say that a proper convex function $f$ is lower-semicontinuous if all of its levelsets ${\{\xx \in \dom(f)\,|\,f(\xx) \leq \alpha \}}$ are closed.
\end{definition}

\vskip 0.5em

\begin{definition}[strong convexity]\label{definition:strong_convexity}
  A function $f$ is said $\mu$-strongly convex if it verifies the following inequality for all $\xx, \yy$ in the domain and any $\uu \in \partial f(\xx)$
  \begin{equation}
    f(\xx) \leq f(\yy) + \langle \uu, \xx - \yy \rangle - \frac{\mu}{2}\|\xx - \yy\|^2
  \end{equation}
\end{definition}

\vskip 0.5em

\begin{definition}[relative interior] The relative interior of a convex set $C \subseteq \RR^p$ is defined as
\begin{equation}
\operatorname{relint}(C) = \{\xx \in C : \forall {\yy \in C} \; \exists {\lambda > 1}: \lambda \xx + (1-\lambda)\yy \in C\}
\end{equation}
\end{definition}

\vskip 0.5em

\begin{lemma}[subgradient characterization of proximal operator]\label{lemma:prox_characterization}
Let $g$ be a convex proper lower semicontinuous function. Then for any $
\xx$ in its domain and any $\gamma > 0$ we have the following characterization of proximal operator:
\begin{equation}
\zz = \prox_{\gamma g}(\xx) \iff \frac{1}{\gamma}(\xx - \zz) \in \partial g(\zz)
\end{equation}
\end{lemma}
\begin{proof}
By the definition of proximal operator we have that $\zz = \prox_{\gamma g}(\xx)$ is equivalent to
\begin{align}
&\zz \in \argmin_{\zz'} g(\zz') + \frac{1}{2\gamma}\| \zz' - \xx\|^2\\
&\iff 0 \in \partial g(\zz) + \frac{1}{\gamma}(\zz - \xx)\\
&\iff \frac{1}{\gamma}(\xx - \zz) \in \partial g(\zz)
\end{align}
where the first equivalence is a consequence of the first order optimality conditions.
\end{proof}

\begin{lemma}[conjugate-inverse identity]\label{lemma:conjugate_inverse}
Let $h$ be a convex, proper lower semicontinuous function. Then
\begin{equation}
\uu \in \partial h(\zz) \iff \zz \in \partial h^*(\uu)~.
\end{equation}
In other words, $(\partial h)^{-1} = \partial h^*$.
\end{lemma}
\begin{proof}
See e.g. \citep[Corollary 16.30]{bauschke2017convex} or \citep[Proposition 11.3]{rockafellar1998variational}.
\end{proof}

\begin{lemma}[Improved Young's inequality] \label{lemma:young}
For all $\boldsymbol a, \boldsymbol b, \boldsymbol c$ we have,
\begin{equation}
    \|\boldsymbol a+\boldsymbol b+ \boldsymbol c\|^2 \leq 3(\|\boldsymbol a\|^2+ \|\boldsymbol b\|^2 + \|\boldsymbol c\|^2)
\end{equation}
\end{lemma}
\proof We develop the squared norm to get,
\begin{align}
    \|\boldsymbol a+\boldsymbol b+ \boldsymbol c\|^2
    & = \|\boldsymbol a\|^2 + 2 \langle \boldsymbol a, \boldsymbol b \rangle +  \|\boldsymbol b\|^2 + 2 \langle \boldsymbol b, \boldsymbol c \rangle + \|\boldsymbol c\|^2 + 2 \langle \boldsymbol a, \boldsymbol c \rangle \\
    & \leq \|\boldsymbol a\|^2 + \|\boldsymbol a\|^2 + \|\boldsymbol b\|^2 +  \|\boldsymbol b\|^2 + \|\boldsymbol b\|^2 + \|\boldsymbol c\|^2 + \|\boldsymbol c\|^2 + \|\boldsymbol a\|^2 + \|\boldsymbol c\|^2 \\
    & \qquad (\text{Young's inequality on } 2 \langle \boldsymbol a, \boldsymbol b \rangle, 2 \langle \boldsymbol b, \boldsymbol c \rangle \text{ and } 2 \langle \boldsymbol a, \boldsymbol c \rangle) \nonumber \\
    & = 3(\|\boldsymbol a\|^2+ \|\boldsymbol b\|^2 + \|\boldsymbol c\|^2)
\end{align}
\endproof
% \clearpage
%
% \section{$k$-operator splitting}\label{apx:k_operator_splitting}

\clearpage
\section{Adaptive (k+1) operator splitting}\label{apx:k_operator_splitting}

For completeness we write below the full adaptive $(k+1)$ operator splitting, which is merely Algorithm~\ref{alg:algo_three_ls} applied to the reformulation of \eqref{eq:obj_fun_k} in \S\ref{scs:extension_k_terms}. It hence solves an optimization problem of the form
\begin{equation}
  \minimize_{\xx \in \RR^p}\,\vphantom{\sum_i^n} \varphi(\xx) + \textstyle\sum_{j=1}^k h_j(\xx) ~,
\end{equation}
where $\varphi$ is $L_f$-smooth and each $h_j$ is proximal.

In this case we name the quadratic used in the sufficient decrease condition $\widetilde{Q}_t$, and is defined as
\begin{align}
  \widetilde Q_t(\xx) \defas \varphi(\overline\ZZ_t) + \langle \nabla \varphi(\overline\ZZ_t), \xx - \overline\ZZ_t\rangle + \frac{1}{2\gamma_t}\|\xx \mathbf{1}_k^T - \ZZ_t\|_F^2~,
\end{align}
where $\|\cdot\|_F$ denotes the Frobenius norm and $\mathbf{1}_k$ is the $k$-dimensional vector of ones.

\begin{algorithm}[H]
  \caption{Algo 1}
  \KwIn{$\zz_0 \in \RR^p$, $\uu_0 \in \RR^p$, $\gamma_0 > 0, \tau \in (0, 1)$}
  $\UU_0 = \uu_0 \boldsymbol{1}_k^T$

  $\ZZ_0 = \zz_0 \boldsymbol{1}_k^T$

  \For{$t=1, \ldots, T$ }{

  \Repeat(\hfill $\triangleright$ step-size search loop){

  $\rr_t = \frac{1}{k}\nabla \varphi(\zz_t)$

  $\xx_{t+1} = \overline{\ZZ}_t - \gamma_t \overline\UU_t - \gamma_t \rr_t$

  $\delta_t = \widetilde{Q}_t(\xx_{t+1}) - \varphi(\xx_{t+1})$

      \eIf{$\delta_t \geq 0 \,$\label{alg:ls_k_2}}{
      {\bfseries break}\hfill $\triangleright$ sufficient decrease verified
      }{
      $\gamma_t = \tau \gamma_t$\hfill $\triangleright$ decrease step-size
      }
  }

  \For{$j=1, \ldots, k\,$}{

  $\ZZ_{j, t+1} = \prox_{\gamma h_j}({\xx}_{t+1} + \gamma_t {\UU}_{j, t})$

  $\UU_{j, t+1} = \UU_{j, t} + (\xx_{t+1} - \zz_{t+1})/ \gamma_t$
  }

  \Variantone{

  $\gamma_{t+1} = \gamma_{t}$\label{line:update_gamma_v1_2}
  }

  % {\bfseries Variant 2} ()
  \Varianttwo(\hfill $\triangleright$ only if each $h_j$ is $\beta_h$-Lipschitz){

  Choose any $\gamma_{t+1} \in [\gamma_t, \sqrt{\gamma_{t}^2 +  \gamma_t \delta_t (2\beta_h)^{-2}}]$\label{line:increase_ls_2}
  }

  }
  % \Return{$\xx_{t+1}$, $\uu_{t+1}$}
  \Return{${\xx}_{t+1}$, ${\UU}_{t+1}$}
\caption{Adaptive $(k+1)$-Operator Splitting}\label{alg:algo_three_ls_k_2}
\end{algorithm}

\paragraph{Discussion.}
The above algorithm requires the storage of a vector of size $p$ and two matrices of size $k \times p$. Hence, for three operators ($k=2$) this requires to store $2 p$ more elements than the three previous formulation, which only requires to store three arrays of size $p$. Hence, for $k=2$, the formulation of Algorithm~\ref{alg:algo_three_ls} should be preferred. Of course, for $k > 2$, it is necessary to use this formulation.

\clearpage

\section{Analysis}\label{apx:analysis}

In this section we provide the proofs for the theorems stated in the Analysis section without proof. The appendix is organized as follows:
\begin{itemize}
  \item In \ref{apx:fixed_point_charac} we prove the fixed point characterization of our three operator splitting variant. Its proof is mostly independent of the other results.
  \item In \ref{apx:key_recursive} we prove an inequality that relates the saddle point suboptimality to the current and previous iterates. This inequality forms the core of both linear and sublinear convergence proofs. Because of its importance we name it ``Key recursive inequality''.
  \item Finally, in \ref{apx:sublinear_convergence} and  \ref{apx:linear_convergence} we make use of the previous ``key recursive inequality'' to prove the sublinear  (Theorem \ref{thm:sublinear} and Corollary \ref{cor:sublinear_convergence}) and linear (Theorem \ref{thm:linear_convergence}) convergence results.
\end{itemize}

Throughout this section we assume that Assumptions 1 and 2 of \S\ref{scs:analysis} are verified without explicit mention.

\subsection{Fixed point characterization}\label{apx:fixed_point_charac}

In this subsection we provide a proof for Theorem~\ref{thm:fixed_point}.
\citet[Lemma 2.2]{davis2017three} proved a weaker result that characterized only the first coordinate: in our notation they proved that if $(\xx, \uu) \in\Fix(\boldsymbol{T}_\gamma)$, then $\xx$ is a primal solution. Our theorem extends this results into a full characterization of the operator.

\begin{customtheorem}{\ref{thm:fixed_point}}
Let $\mathcal{P}^\star$ denote the set of minimizers of the primal objective~\eqref{eq:primal_loss} and $\mathcal{D}^\star$ the set of minimizers of the dual objective~\eqref{eq:dual_loss}\,. Then the set of fixed points of $\TT_\gamma$ is given by
\begin{equation}
\Fix(\boldsymbol{T}_{\gamma}) = \mathcal{P}^\star \times \mathcal{D}^\star~.
\end{equation}
\end{customtheorem}

\begin{proof}
We will find it useful to first characterize the fixed points of $\boldsymbol{T}_{\gamma}$ by a subdifferential inclusion. Consider the following set of equivalences

\begin{align}
 (\zz, \uu) \in \Fix(\TT_\gamma) &\iff (\zz, \uu) = \TT_\gamma(\zz, \uu)\label{eq:fixed_point_initial}\\
    &\iff \begin{cases}
    \zz = \prox_{\gamma h}(\xx(\zz, \uu) + \gamma \uu)\\
    \uu = \uu + (\xx(\zz, \uu) - \zz)/ \gamma \\
    \text{ with } \xx(\zz, \uu) = \prox_{\gamma g}(\zz - \gamma(\uu + \nabla f(\zz)))
    \end{cases}
    &\text{ (by definition of $\boldsymbol{T}_{\gamma}$)}\\
    &\iff \begin{cases}
    \zz = \prox_{\gamma g}(\zz - \gamma(\uu + \nabla f(\zz)))\\
    \zz = \prox_{\gamma h}(\zz + \gamma \uu)\\
    \end{cases}
    &\text{ ($\zz = \xx(\zz, \uu)$ by second equation)}\\
  &\iff \begin{cases} - \uu - \nabla f(\zz) \in \partial g(\zz) \label{eq:t13}\\
  \uu \in \partial h(\zz)\end{cases}&\text{ (by Lemma~\ref{lemma:prox_characterization})}\\
  &\iff \begin{cases} - \uu  \in \partial (f + g)(\zz) \\
  \uu \in \partial h(\zz)\end{cases}&\label{eq:t11}\\
  % &\iff \begin{cases} -\frac{1}{\gamma}(\yy - \zz) \in  \partial (f + g)(\zz)\\
  % \frac{1}{\gamma}(\yy - \zz) \in \partial h(\zz)
  % \end{cases} \label{eq:t14}\\
  &\iff \begin{cases} \zz \in  \partial (f + g)^*(-\uu)\\
  \zz \in \partial h^*(\uu )
  \end{cases}& \text{(by Lemma~\ref{lemma:conjugate_inverse})}\label{eq:t15}
  % &\iff \zz \in \partial(f + g)^{*}(-\balpha)\label{eq:t15}~.
\end{align}

The rest of the proof is divided two parts, proving in the first part $\Fix(\boldsymbol{T}_{\gamma}) \subseteq {\mathcal{P}^\star \times \mathcal{D}^\star}\,$, and the reverse inclusion in the second part.

\emph{Part 1}. Our goal is to prove $\Fix(\boldsymbol{T}_{\gamma}) \subseteq {\mathcal{P}^\star \times \mathcal{D}^\star}\,$. Let $(\zz, \uu) \in\Fix(\boldsymbol{T}_{\gamma})$.
Adding together the equations in \eqref{eq:t11} one obtains
\begin{equation}\label{eq:first_order_optimality_primal}
0 \in   \partial h(\zz) + \partial g(\zz) + \nabla f(\zz)~,
\end{equation}
and so $\zz$ is a minimizer of the primal objective. Subtracting the equations in~\eqref{eq:t15} we also have
\begin{equation}\label{eq:first_order_optimality_dual}
0 \in  \partial h^*(\balpha) - \partial(f + g)^{*}(-\balpha)~,
\end{equation}
and so $\uu$ is a minimizer of the dual objective. We have proved $\Fix(\boldsymbol{T}_{\gamma}) \subseteq \mathcal{P}^\star \times \mathcal{D}^\star\,$.

\vspace{0.5em}\emph{Part 2}. Our goal now is to prove the inverse inclusion, ${\mathcal{P}^\star \times \mathcal{D}^\star} \subseteq \Fix(\boldsymbol{T}_{\gamma})$. Let $(\zz, \balpha) \in \mathcal{P}^\star \times \mathcal{D}^\star$, we will prove that $(\zz, \balpha)$ is a fixed point of $\TT_\gamma$.

We start by recalling the notion of \emph{paramonotonicity}, which will play a key role in this proof. This notion was introduced by \citet{Iusem1998} and is key to characterizing the set of fixed points of related methods, such as the Douglas-Rachford splitting~\citep{bauschke2012attouch}. An operator $\CC$ is said to be paramonotonic if the following implication is verified
\begin{equation}
\begin{rcases*}
\widetilde{\boldsymbol{a}} \in \CC \boldsymbol{a}\\
\widetilde{\boldsymbol{b}} \in \CC \boldsymbol{b}\\
\langle \widetilde{\boldsymbol{a}} - \widetilde{\boldsymbol{b}}, \boldsymbol{a} - \boldsymbol{b}\rangle = 0
\end{rcases*} \implies \widetilde{\boldsymbol{a}} \in \CC \boldsymbol{b} \text{ and } \widetilde{\boldsymbol{b}} \in \CC \boldsymbol{a}~.
\end{equation}

It is known that the subdifferential of a convex proper lower semicontinuous function is paramonotonic~\citep[Proposition 2.2]{Iusem1998}. Hence we have that $\partial h$ and $\partial (f + g)$ are paramonotonic.

By the first-order optimality conditions on the primal and dual loss we have that there exists elements $\uu_\zz$ and $\zz_\uu$ such that
\begin{gather}
\uu_\zz \in \partial h(\zz) \cap (-\partial (f + g)(\zz))\label{eq:u_z_inclusion}\\
\zz_\uu \in \partial h^*(\uu) \cap (\partial (f + g)^*(-\uu))~,
\end{gather}
where the second inclusion can be written equivalently using the conjugate-inverse identity (Lemma~\ref{lemma:conjugate_inverse}) as
\begin{equation}
\uu \in \partial h(\zz_\uu) \cap (- \partial (f + g)(\zz_\uu))~.\label{eq:u_inclusion}
\end{equation}
Using Eq.~\eqref{eq:u_z_inclusion} and \eqref{eq:u_inclusion} we have by monotony of $\partial h$ and $\partial (f+g)$
\begin{equation}
\langle \uu_\zz - \uu, \zz - \zz_\uu \rangle\geq 0 ~\text{ and } \langle \uu_\zz - \uu, \zz - \zz_\uu \rangle\leq 0
\end{equation}
from where we necessarily have $\langle \uu_\zz - \uu, \zz - \zz_\uu \rangle = 0$. We hence have by paramonotonicity of $\partial h$
\begin{equation}
\begin{rcases}
\uu_\zz \in \partial h(\zz)\\
\uu \in \partial h(\zz_\uu)\\
\langle \uu_\zz - \uu, \zz - \zz_\uu\rangle = 0
\end{rcases} \implies \uu \in \partial h(\zz)
\end{equation}
Similarly, by paramonotonicity of $\partial (f+g)$ we have
\begin{equation}
\begin{rcases}
-\uu_\zz \in \partial (f+g)(\zz)\\
-\uu \in \partial (f+g)(\zz_\uu)\\
\langle \uu_\zz - \uu, \zz - \zz_\uu\rangle = 0
\end{rcases} \implies -\uu \in \partial (f + g)(\zz)
\end{equation}
Combining the last two equations we have by the definition of $\yy$ the following inclusions
\begin{equation}
\begin{cases}-\uu \in \partial (f+g)(\zz)\\
\uu \in h(\zz)~.\end{cases}
\end{equation}
These are the same subdifferential inclusions of \eqref{eq:t11}, which are equivalent to $(\zz, \uu) \in \Fix(\boldsymbol{T}_{\gamma})$ (Eq.~\eqref{eq:fixed_point_initial}). This concludes the proof.
\end{proof}

\vspace{1em}

% \begin{customcorollary}{X}
% \corllaryonetext
% \end{customcorollary}
% \begin{proof}
% This is a consequence of Eq.~\eqref{eq:first_order_optimality_primal} and Eq.~\eqref{eq:first_order_optimality_dual}.
% \end{proof}

\vspace{2em}

\subsection{Key recursive inequality}\label{apx:key_recursive}

In this subsection we present a lemma that will be key for further proofs, as it relates the saddle point suboptimality to the current and previous iterates.

\begin{lemma}[Key recursive inequality]\label{lemma:key_recursive}
Let $f$ be $\mu_f$-strongly convex and $h^*$ be $\mu_h$-strongly convex (where we allow $\mu_f = 0$ and/or $\mu_h = 0$). Then after $t$ iterations of Algorithm~\ref{alg:algo_three_ls} we have the following inequality for all $(\xx, \balpha)$ in the domain of $\mathcal{L}$, with $\gamma_{-1} = \gamma_0\,,~\delta_{-1} = 0$.

For Variant 1:
\begin{equation}\label{eq:recursive_ineq_1}
2\gamma_t(\mathcal{L}(\xx_{t+1}, \balpha) - \mathcal{L}(\xx, \uu_{t+1})) +  \|\zz_{t+1} - \xx\|^2 +  (1 - \sigma)^{-1}\|\gamma_{t}(\uu_{t+1} - \uu)\|^2\leq (1 - \rho)\|\zz_{t} - \xx\|^2  + \|\gamma_{t-1}(\uu_{t}\!-\!\uu)\|^2~.
\end{equation}

For Variant 2:
\begin{equation}\label{eq:recursive_ineq_2}
\begin{aligned}
2\gamma_t(\mathcal{L}(\xx_{t+1}, \balpha) - \mathcal{L}(\xx, \uu_{t+1}))& +  \|\zz_{t+1} - \xx\|^2 + (1 - \xi)^{-1}\|\gamma_t(\uu_{t+1} - \uu)\|^2\\
&\qquad\leq(1 - \rho)\|\zz_{t} - \xx\|^2 +\|\gamma_{t-1}(\uu_{t}\! -\! \uu)\|^2
 + \gamma_{t-1} \delta_{t-1} - 2 \gamma_t \delta_t    ~.
\end{aligned}
\end{equation}
\end{lemma}

\begin{proof}
The proof is structured in two parts. In the first part, we will bound the ``primal'' suboptimality ${\mathcal{L}(\xx_{t+1}, \balpha_{t+1}) - \mathcal{L}(\xx, \balpha_{t+1})}$ and in the second part we will bound the ``dual'' suboptimality ${\mathcal{L}(\xx_{t+1}, \balpha) - \mathcal{L}(\xx_{t+1}, \balpha_{t+1})}$. Finally, adding both will yield the desired result.

\underline{\emph{Part 1}}. By $\mu_f$-strong convexity $f$ verifies the following inequality for an arbitrary $\xx$
\begin{equation}
f(\zz_{t}) - f(\xx)\leq \langle \nabla f(\zz_{t}), \zz_{t} - \xx\rangle - \frac{\mu_f}{2}\|\zz_{t} - \xx\|^2~. \label{eq:f_convex}
\end{equation}
By the line search condition and the definition of $\delta_t$ (we recall $\delta_t = Q_t(\xx_{t+1}) - f(\xx_{t+1})$) we have at each iteration $f(\xx_{t+1}) = Q_t(\xx_{t+1}) - \delta_t$. This gives by definition of $Q_t$:
\begin{equation}
f(\xx_{t+1}) =  f(\zz_{t}) + \langle \nabla f(\zz_{t}), \xx_{t+1} - \zz_{t} \rangle + \frac{1}{2 \gamma_t}\|\zz_{t}-\xx_{t+1}\|^2 - \delta_t~\\
\end{equation}
Adding the previous two equations gives  the following inequality, which we will make use of later on
\begin{equation}\label{eq:convex_and_l_smooth}
f(\xx_{t+1}) - f(\xx)\,\leq \langle \nabla f(\zz_{t}), \xx_{t+1} - \xx \rangle\, + \frac{1}{2 \gamma_t}\|\zz_{t}-\xx_{t+1}\|^2- \frac{\mu_f}{2}\|\zz_{t} - \xx\|^2- \delta_t~.
\end{equation}

From the subdifferential characterization of the proximal operator (Lemma~\ref{lemma:prox_characterization}), the update $\xx_{t+1} = {\prox_{\gamma_t g}(\zz_{t} - \gamma_t (\uu_{t} + \nabla f(\zz_{t})))}$ of Line~\ref{line:x_update} implies the following subdifferential inclusion
  \begin{equation}
  \frac{1}{\gamma_t}(\zz_{t} - \gamma_t(\uu_t + \nabla f(\zz_{t})) - \xx_{t+1}) \in \partial g(\xx_{t+1})~.
\end{equation}
By of $g$ we then have the following inequality
\begin{align}
    g(\xx_{t+1}) - g(\xx) &\leq \frac{1}{\gamma_{t}} \langle \zz_{t}  - \gamma_t(\uu_{t} +\nabla f(\zz_{t}))- \xx_{t+1}, \xx_{t+1} - \xx\rangle \\
    &=\frac{1}{\gamma_{t}} \langle \zz_{t}  - \xx_{t+1}, \xx_{t+1} - \xx\rangle - \langle \uu_{t} +\nabla f(\zz_{t}), \xx_{t+1} - \xx \rangle~.\label{eq:ineq_g_convexity}
\end{align}
Adding together \eqref{eq:convex_and_l_smooth} and \eqref{eq:ineq_g_convexity} we obtain
\begin{equation}\label{eq:primal_part_ineq}
\begin{aligned}
f(\xx_{t+1}) + g(\xx_{t+1}) - f(\xx) - g(\xx) &\leq \frac{1}{\gamma_{t}} \langle \zz_{t}  - \xx_{t+1}, \xx_{t+1} - \xx\rangle + \frac{1}{2 \gamma_t}\|\zz_{t}-\xx_{t+1}\|^2\\
&\qquad\qquad- \langle \uu_{t}, \xx_{t+1} - \xx\rangle- \frac{\mu_f}{2}\|\zz_{t} - \xx\|^2 - \delta_t~.
\end{aligned}\end{equation}
We will now use this last inequality to bound $\mathcal{L}(\xx_{t+1}, \balpha_t) - \mathcal{L}(\xx, \balpha_t)$.
For this, we will make extensive of the cosine identity $2\langle \boldsymbol a, \boldsymbol b\rangle = {\|\boldsymbol{a} + \boldsymbol{b}\|^2} - \|\boldsymbol a\|^2 - \|\boldsymbol b\|^2$ and the
primal-dual relationship $\gamma_t (\uu_{t+1} - \uu_{t}) = \xx_{t+1} - \zz_t$, which is an immediate consequence of the
definition of $\uu_{t+1}$ in Line~\ref{line:u_update}.
\begin{align}
&\mathcal{L}(\xx_{t+1}, \uu_{t+1}) - \mathcal{L}(\xx, \uu_{t+1}) = f(\xx_{t+1})  + g(\xx_{t+1}) - f(\xx) - g(\xx) + \langle \xx_{t+1} - \xx, \uu_{t+1}\rangle\\
&\qquad\stackrel{\eqref{eq:primal_part_ineq}}{\leq}  \frac{1}{\gamma_t}\langle \zz_{t} - \xx_{t+1}, \xx_{t+1} - \xx\rangle + \frac{1}{2 \gamma_t }\|\zz_{t} - \xx_{t+1}\|^2  + \langle \xx_{t+1} - \xx, \uu_{t+1} - \uu_{t}\rangle  - \frac{\mu_f}{2}\|\zz_{t} - \xx\|^2 - \delta_t\\
&\qquad= \left(\frac{1 - \gamma_t\mu_f}{2 \gamma_t}\right)\|\zz_{t} - \xx\|^2 - \frac{1}{2\gamma_t}\|\xx_{t+1} - \xx\|^2 - \frac{1}{\gamma_t}\langle \xx_{t+1} - \xx, \zz_{t+1} - \xx_{t+1}\rangle - \delta_t \\
&\qquad\qquad \text{ (cosine identity on first term and primal-dual relationship on second-last term)}\nonumber\\
&\qquad= \left(\frac{1 - \gamma_t\mu_f}{2 \gamma_t}\right)\|\zz_{t} - \xx\|^2 - \frac{1}{2\gamma_t}\|\zz_{t+1} - \xx\|^2  + \frac{1}{2\gamma_t}\|\zz_{t+1} - \xx_{t+1}\|^2 - \delta_t\label{eq:primal_suboptimality}\\
&\qquad \qquad\text{ (cosine identity on the second-last term)}\nonumber
\end{align}

\vspace{0.5em} \underline{\emph{Part 2}}. From the subdifferential characterization of the proximal operator (Lemma~\ref{lemma:prox_characterization}), the update $\zz_{t+1} = {\prox_{\gamma_t h}(\xx_{t+1} + \gamma_t \uu_{t})}$ of Line~\ref{line:z_update} we have the following subdifferential inclusion
\begin{equation}
  \uu_{t+1} = \frac{1}{\gamma_t}(\xx_{t+1} - \zz_{t+1}) + \uu_{t} \in \partial h(\zz_{t+1}) \implies
    \zz_{t+1} \in \partial h^*(\uu_{t+1})~,
\end{equation}
where the implication is a consequence of the conjugate-inverse identity (Lemma~\ref{lemma:conjugate_inverse}). As we did before, this inclusion can be used to obtain an inequality in terms of the function values. By $\mu_h$-strong convexity of $h^*$ we have
\begin{equation}\label{eq:dual_part_ineq}
  h^*(\uu_{t+1}) - h^*(\uu) \leq \langle \zz_{t+1}, \uu_{t+1} - \uu \rangle - \frac{\mu_h}{2}\|\uu_{t+1} - \uu\|^2~.
\end{equation}
We can now use this inequality to to bound $\mathcal{L}(\xx_t, \balpha) - \mathcal{L}(\xx_t, \uu_{t+1})$ as follows:
\begin{align}
&\mathcal{L}(\xx_{t+1}, \uu) - \mathcal{L}(\xx_{t+1}, \uu_{t+1}) = h^*(\uu_{t+1}) - h^*(\uu) + \langle \xx_{t+1}, \uu - \uu_{t+1} \rangle\\
&\qquad\stackrel{\eqref{eq:dual_part_ineq}}{\leq} \langle \zz_{t+1} - \xx_{t+1}, \uu_{t+1} - \balpha\rangle - \frac{\mu_h}{2}\|\uu_{t+1} - \uu\|^2\\
&\qquad= \frac{1}{\gamma_t}\langle \zz_{t+1} - \xx_{t+1}, \gamma_t(\uu_{t+1} - \balpha)\rangle  - \frac{\mu_h}{2}\|\uu_{t+1} - \uu\|^2\\
&\qquad=  \frac{1}{\gamma_t}\langle \zz_{t+1} - \xx_{t+1}, \gamma_t(\uu_{t+1} - \balpha)\rangle  - \frac{\gamma_t^{-1}\mu_h}{2\gamma_t}\|\gamma_t(\uu_{t+1} - \uu)\|^2\\
&\qquad= \frac{1}{2\gamma_t}\|\gamma_t(\uu_{t} - \uu)\|^2 - \frac{1}{2\gamma_t}\|\zz_{t+1} - \xx_{t+1}\|^2 - \Big(\frac{1 +\gamma_t^{-1}\mu_h}{2\gamma_t}\Big) \|\gamma_t(\uu_{t+1} - \uu)\|^2~,\label{eq:dual_suboptimality}
\end{align}
where in the second line we have used again the primal-dual relationship $\xx_{t+1} - \zz_{t+1} = \gamma_t (\uu_{t+1} - \uu_{t})$.

\vspace{0.5em} \underline{\emph{Third part: putting it all together}}. Adding the inequalities from Eq.~\eqref{eq:primal_suboptimality} and Eq.~\eqref{eq:dual_suboptimality} and multiplying everything by by $2\gamma_t$ we obtain
\begin{align}
2 \gamma_t(\mathcal{L}(\xx_{t+1}, \balpha) - \mathcal{L}(\xx, \uu_{t+1})) &\leq (1 - \gamma_t \mu_f)\|\zz_{t} - \xx\|^2 - \|\zz_{t+1} - \xx\|^2\nonumber\\
&\qquad  + \|\gamma_t(\uu_{t} - \uu)\|^2  - (1 +\gamma_t^{-1}\mu_h)\|\gamma_t(\uu_{t+1} - \uu)\|^2 - 2 \gamma_t \delta_t~.
\end{align}
We will now prove that $\gamma_t$ is lower bounded by $\min\{\tau L_f^{-1}, \gamma_0\}$ through a distinction of cases. If $\gamma_0 \geq \tau/L_f$, then by the properties of $L_f$-smooth functions, the sufficient decrease condition is verified for all step-size smaller than $1/L_f$. The resulting step-size can still be smaller than this quantity if the sufficient decrease condition fails for some step-size $\gamma> 1/L_f$ and $\tau$ is small enough such that $\tau \gamma \leq 1/L_f$. Even in this worst case scenario we have $\gamma_t \geq \tau/L_f$, which proves the bound for the case $\gamma_0 \geq \tau/L_f$. If $\gamma_0 \leq \tau/L_f$, then the sufficient decrease condition is verified at this first iterate, and the step-size can only increase.

We can hence bound $\gamma_t\mu_f$ by $\rho$, defined in Eq. \eqref{eq:def_linear_rates}, to obtain
\begin{align}\label{eq:recursive_tmp_1}
2 \gamma_t(\mathcal{L}(\xx_{t+1}, \balpha) - \mathcal{L}(\xx, \uu_{t+1})) &\leq (1 - \rho)\|\zz_{t} - \xx\|^2 - \|\zz_{t+1} - \xx\|^2\nonumber\\
&\qquad  + \|\gamma_t(\uu_{t} - \uu)\|^2  - (1 +\gamma_t^{-1}\mu_h)\|\gamma_t(\uu_{t+1} - \uu)\|^2 - 2 \gamma_t \delta_t~.
\end{align}

For Variant 1, we can just drop the non-positive term $-2 \gamma_t\delta_t$ and use the non-increasing step to bound the following terms as:
\begin{align}
\|\gamma_t(\uu_{t} - \uu)\|^2 &\leq \|\gamma_{t-1}(\uu_{t} - \uu)\|^2 \quad \text{ (with $\gamma_{-1} = \gamma_0$ by definition)}\\
- (1 +\gamma_t^{-1}\mu_h) &\leq - (1 +\gamma_0^{-1}\mu_h)\\
&= - \frac{\gamma_0 + \mu_h}{\gamma_0} = - (\frac{\gamma_0+ \mu_h- \mu_h}{\gamma_0 + \mu_h})^{-1}\\
&= -(1 - \frac{\mu_h}{\gamma_0 + \mu_h})^{-1} = - (1 - \sigma)^{-1}
\end{align}
Replacing in \eqref{eq:recursive_tmp_1}  gives
\begin{align}
2 \gamma_t(\mathcal{L}(\xx_{t+1}, \balpha) - \mathcal{L}(\xx, \uu_{t+1})) &\leq (1 - \rho)\|\zz_{t} - \xx\|^2 - \|\zz_{t+1} - \xx\|^2\nonumber\\
&\qquad  + \|\gamma_t(\uu_{t} - \uu)\|^2  - (1 -\sigma)^{-1}\|\gamma_t(\uu_{t+1} - \uu)\|^2 ~,
\end{align}
from where we obtain the desired bound \eqref{eq:recursive_ineq_1} by reordering the terms.

%
% we can replace $\gamma_{t}$ by $\gamma_{t-1}$ in the second-last term, which gives the desired result.

We will now derive a similar bound for Variant 2. In this case, the $\beta_h$-Lipschitz assumption on $h$ implies that the norm of every element in $\dom h^*$ is bounded by $\beta_h$ (see e.g., \citep[Corollary 13.3.3]{rockafellar1997convex}). This way we bound $\|\uu_t - \uu\|^2 \leq 2 \|\uu_t\|^2 + 2 \|\uu\|^2 \leq 4 \beta_h$.

Assuming first $t > 0$, we have the following sequence of inequalities for any $\gamma_t \geq \gamma_{t-1}$:
\begin{align}
  \|\gamma_{t}(\uu_{t} - \uu)\|^2 - \|\gamma_{t-1}(\uu_{t} - \uu)\|^2 -  \gamma_{t-1} \delta_{t-1} &= (\gamma_{t}^2  - \gamma_{t-1}^2)\|\uu_{t} - \uu\|^2 -  \gamma_{t-1} \delta_{t-1}\\
  &\leq  (\gamma_{t}^2  - \gamma_{t-1}^2)4\beta_h^2 -  \gamma_{t-1} \delta_{t-1}\\
  & \qquad \text{ (using $\gamma_t \geq \gamma_{t-1}$ and that $h$ is $\beta_h$-Lipschitz)}\nonumber \\
  &\leq (\gamma_{t-1}^2 +  \gamma_{t-1} \delta_{t-1} \beta_h^{-2}/4  - \gamma_{t-1}^2)4\beta_h^2 -  \gamma_{t-1} \delta_{t-1}\\
  &\qquad \text{ (by the choice of $\gamma_{t+1}$ in Line \ref{line:increase_ls})}\nonumber\\
  &= \gamma_{t-1}^2 \beta_h^2 +  \gamma_{t-1} \delta_{t-1} - \gamma_{t-1}^2\beta_h^2 - \gamma_{t-1} \delta_{t-1}\\
  &= 0~,
\end{align}
which reordering gives
\begin{equation}\label{eq:variant2_tmp_inequality}
  \|\gamma_{t}(\uu_{t} - \uu)\|^2 \leq \|\gamma_{t-1}(\uu_{t} - \uu)\|^2 +  \gamma_{t-1} \delta_{t-1}~.
\end{equation}
This inequality is also trivially true when $\gamma_t\leq \gamma_{t-1}$ because of the non-negativity of $\delta_{t}$.
If $t=0$, then we have by definition $\gamma_0 = \gamma_{-1}, \delta_{-1} = 0$ and so the above inequality is also trivially verified.

Finally, plugging this lasts bound in~\eqref{eq:recursive_tmp_1} gives
\begin{equation}\begin{aligned}
2 \gamma_t(\mathcal{L}(\xx_{t+1}, \balpha) - \mathcal{L}(\xx, \uu_{t+1}))
&\leq (1 - \rho)\|\zz_{t} - \xx\|^2 - \|\zz_{t+1} - \xx\|^2\nonumber\\
&\quad  + \|\gamma_{t-1}(\uu_{t} - \uu)\|^2  - (1+ \gamma_t^{-1}\mu_h)\|\gamma_t(\uu_{t+1} - \uu)\|^2 - 2 \gamma_t \delta_t +  \gamma_{t-1}\delta_{t-1}~.
\end{aligned}\end{equation}
Now we will use the following bound on $\gamma_t^{-1}$: $\mu_f \leq \gamma_t^{-1}$. This bound can be deduced from the the strong convexity inequality (Definition~\ref{definition:strong_convexity}), as otherwise the sufficient decrease condition $\delta_t \geq 0$ would not hold. In all, we have the desired bound
\begin{equation}\begin{aligned}
2 \gamma_t(\mathcal{L}(\xx_{t+1}, \balpha) - \mathcal{L}(\xx, \uu_{t+1})) &\leq (1 - \rho)\|\zz_{t} - \xx\|^2 - \|\zz_{t+1} - \xx\|^2 + \|\gamma_{t-1}(\uu_{t} - \uu)\|^2\nonumber\\
&\qquad    - (1 +\mu_f\mu_h)\|\gamma_t(\uu_{t+1} - \uu)\|^2  +  \gamma_{t-1} \delta_{t-1} - 2 \gamma_t \delta_t~.
\end{aligned}\end{equation}
Finally, by trivial algebraic manipulations we have
\begin{equation}
  (1 - \xi)^{-1} = (1 - \frac{\mu_f\mu_h}{1 + \mu_f\mu_h})^{-1} = (\frac{1}{1 + \mu_f\mu_h})^{-1} = {1 + \mu_f\mu_h}
\end{equation}
which is the term that multiplies $\|\gamma_t(\uu_{t+1} - \uu)\|^2$ in the previous equation. Replacing with $(1 - \xi)^{-1}$ in that equation we have the desired bound.

\end{proof}

\vspace{2em}

\subsection{Sublinear convergence}\label{apx:sublinear_convergence}

\begin{customtheorem}{\ref{thm:sublinear}}[sublinear convergence rate]
For every iteration ${t \geq 0}$ and any $(\xx, \balpha)$ in the domain of $\mathcal{L}$ we have the following convergence rate for Algorithm~\ref{alg:algo_three_ls} (both variants):
\begin{equation}
\mathcal{L}(\overline\xx_{t+1}, \uu)- \mathcal{L}(\xx, \overline\uu_{t+1}) \leq \frac{{\|\zz_0 - \xx\|^2} + {\gamma_0^2\|\uu_0 - \uu\|^2}}{2 s_t}~.
\end{equation}
\end{customtheorem}

\begin{proof}
Adding the Equation of Lemma~\ref{lemma:key_recursive} with $\mu_f=\mu_h = 0$ (which implies $\rho=\sigma=\xi=0$) from $0$ to $t$ and dropping positive terms in the left hand side we get for both Variant 1:
\begin{align}
\sum_{i=0}^t \gamma_i\left(\mathcal{L}(\xx_{i+1}, \balpha) - \mathcal{L}(\xx, \uu_{i+1})\right) &\leq \frac{1}{2}\|\zz_0 - \xx\|^2 + \frac{1}{2}\|\gamma_0(\uu_0 - \uu)\|^2\nonumber\\
&\qquad- \|\gamma_t(\uu_{t+1} - \uu)\|^2 - \|\zz_{t+1} - \xx\|^2\\
&\leq \frac{1}{2}\|\zz_0 - \xx\|^2 + \frac{1}{2}\|\gamma_0(\uu_0 - \uu)\|^2~.
\end{align}
Similarly, for Variant 2 we have
\begin{align}
\sum_{i=0}^t \gamma_i\left(\mathcal{L}(\xx_{i+1}, \balpha) - \mathcal{L}(\xx, \uu_{i+1})\right) &\leq \frac{1}{2}\|\zz_0 - \xx\|^2 + \frac{1}{2}\|\gamma_0(\uu_0 - \uu)\|^2 - \|\gamma_t(\uu_{t+1} - \uu)\|^2\nonumber\\
&\qquad \qquad  - \|\zz_{t+1} - \xx\|^2 + \gamma_{-1}\delta_{-1} - \sum_{i=0}^t \gamma_t \delta_t \\
&\leq \frac{1}{2}\|\zz_0 - \xx\|^2 + \frac{1}{2}\|\gamma_0(\uu_0 - \uu)\|^2  + \gamma_{-1}\delta_{-1} - \sum_{i=0}^t \gamma_t \delta_t\\
&\leq \frac{1}{2}\|\zz_0 - \xx\|^2 + \frac{1}{2}\|\gamma_0(\uu_0 - \uu)\|^2~,
\end{align}
where the last line follows from the definition $\delta_{-1} = 0$ and the non-negativity of $\delta_t$.

In all, we have for both variants
\begin{equation}
\sum_{i=0}^t \gamma_i\left(\mathcal{L}(\xx_{i+1}, \balpha) - \mathcal{L}(\xx, \uu_{i+1})\right) \leq \frac{1}{2}\|\zz_0 - \xx\|^2 + \frac{1}{2}\|\gamma_0(\uu_0 - \uu)\|^2~.\label{eq:sum_suboptimality}
\end{equation}

For a fixed $(\xx, \uu)$, the saddle point suboptimality $\mathcal{L}(\xx_t, \balpha) - \mathcal{L}(\xx, \uu_t)$ is convex in $\xx_t$ and $\uu_t$. By Jensens inequality we can then bound the left hand side of Eq.~\eqref{eq:sum_suboptimality} as
\begin{equation}
\frac{1}{s_t}\sum_{i=0}^t \gamma_i(\mathcal{L}(\xx_{i+1}, \balpha) - \mathcal{L}(\xx, \uu_{i+1}))  \geq \mathcal{L}(\overline\xx_{t+1}, \uu) - \mathcal{L}(\xx, \overline\uu_{t+1})~.
\end{equation}
Combining this last inequality with Eq.~\eqref{eq:sum_suboptimality} we obtain
\begin{equation}
\mathcal{L}(\overline\xx_{t+1}, \uu) - \mathcal{L}(\xx, \overline\uu_{t+1}) \leq \frac{1}{2 s_{t}} \left( \|\zz_0 - \xx\|^2 + \|\gamma_0(\uu_0 - \uu)\|^2\right)~,
\end{equation}
which is the desired result.
\end{proof}

\vspace{2em}

\begin{customcorollary}{\ref{cor:sublinear_convergence}} Let $h$ be $\beta_h$-Lipschitz. Then, we have the following rate for the weighted ergodic iterate
\begin{equation}
P(\overline\xx_{t+1})\!-\!P(\xx^\star) \leq  \frac{\|\zz_0 - \xx^\star\|^2 + 2 \gamma_0^2(\|\uu_0\|^2 + \beta_h^2)}{2 s_t}
\end{equation}
\end{customcorollary}
\begin{proof}
  Let $\widehat\uu \defas \argmin_{\uu} \mathcal{L}(\overline\xx_{t+1}, \uu)$ and $(\xx^\star, \uu^\star)$ be a saddle point of $\mathcal{L}$. Then $\mathcal{L}(\overline\xx_{t+1}, \widehat\uu) = P(\overline\xx_{t+1})$ and $\mathcal{L}(\xx^\star, \uu^\star) = P(\xx^\star)$ by definition of Fenchel dual.

Using this and the previous theorem we can write the following set of inequalities
\begin{align}
  P(\overline\xx_{t+1}) - P(\xx^\star) &= \mathcal{L}(\overline\xx_{t+1}, \widehat\uu) - \mathcal{L}(\xx^\star, \uu^\star)\\
  &\leq \mathcal{L}(\overline\xx_{t+1}, \widehat\uu) - \mathcal{L}(\xx^\star, \widehat\uu)\\
  &\quad \text{ (definition of saddle point, Eq.~\eqref{eq:saddle_point} with $\xx=\xx^\star$) }\nonumber\\
  &\leq \frac{1}{2 s_t} \left(\|\zz_0 - \xx\|^2 + \|\gamma_0(\uu_0 -  \widehat\uu)\|^2\right)\\
  &\quad\text{ (Theorem~\ref{thm:sublinear} with $\xx = \xx^\star, \uu = \widehat\uu$)}\nonumber
\end{align}
The $\beta_h$-Lipschitz assumption on $h$ implies that the norm of every element in $\dom h^*$ is bounded by $\beta_h$ (see e.g., \citep[Corollary 13.3.3]{rockafellar1997convex}). This way we bound $\|\gamma_0(\uu_0 -  \widehat\uu)\|^2 \leq 2 \gamma_0^2\|\uu_0\|^2 + 2 \gamma_0^2\| \widehat\uu\|^2 \leq  2\gamma_0^2(\|\uu_0\|^2 + \beta_h^2)$. Plugging this bound into the last inequality we have the desired bound
\begin{equation}
  P(\overline\xx_{t+1}) - P(\xx^\star) \leq \frac{\|\zz_0 - \xx^\star\|^2 + \gamma_0^2(\|\uu_0\|^2 + \beta_h^2)}{2 s_t}~.
\end{equation}
\end{proof}

\vspace{2em}

\subsection{Linear convergence}\label{apx:linear_convergence}

In this subsection we assume that $f$ is $\mu_f$-strongly convex and $h$ is $L_h$-smooth (with $\mu_f > 0, 0 < L_h < +\infty$). We denote by $\xx^\star$ the minimizer of the primal loss (unique by strong convexity of $P$) and by $\uu^\star$ the minimizer of the dual loss (also unique by strong convexity of $D$, which is a consequence of the duality between $L$-smoothness and strong convexity).

\begin{customtheorem}{\ref{thm:linear_convergence}}
Let $\xx_{t+1}, \uu_{t+1}$ be the iterates  produced by Algorithm~\ref{alg:algo_three_ls} after $t$ iterations. Then we have the following linear convergence for Variant 1 (V1) and Variant 2 (V2):
% \begin{itemize}[leftmargin=*]
% \item For :
\begin{align}
  &\text{V1}: \|\xx_{t+1} - \xx^\star\|^2 \leq \Big(1 - \min\big\{\rho, \sigma\big\}\Big)^{t+1} D_0\\
  &\text{V2}: \|\xx_{t+1} - \xx^\star\|^2 \leq \Big(1 - \min\big\{\rho,\xi, \mfrac{1}{2}\big\}\Big)^{t+1} E_0~,
\end{align}
with $D_0 \defas 6\|\zz_0 - \xx^\star\|^2 + \frac{6}{1 - \sigma}\|\gamma_0(\uu_0 - \uu^\star)\|^2$
 and $E_0 \defas 6\|\zz_0 - \xx^\star\|^2 + \frac{6}{1 - \xi}\|\gamma_0(\uu_0 - \uu^\star)\|^2$.
\end{customtheorem}

\begin{proof}
\textbf{Variant 1:}
By the duality between Lipschitz gradient and strong convexity, $h$ being $L_h$-smooth implies that $h^*$ is $L_h^{-1}$-strongly convex.
Applying Lemma~\ref{lemma:key_recursive} with $\xx = \xx^\star$, $\uu = \uu^\star$ we obtain the following inequality for Variant 1:
\begin{align}
  2 \gamma_t(\mathcal{L}(\xx_{t+1}, \uu^\star) - &\mathcal{L}(\xx^\star, \uu_{t+1})) + \|\zz_{t+1} - \xx^\star\|^2 + (1 - \sigma)^{-1}\|\gamma_{t} (\uu_{t+1} - \uu^\star)\|^2\\
  &\qquad\qquad\leq (1 - \rho) \|\zz_{t} - \xx^\star\|^2 +  \|\gamma_{t-1}(\uu_{t} - \uu^\star)\|^2~,
\end{align}
where we note that the assumption $0 < L_h$ implies $\sigma < 1$ and so $(1 - \sigma)^{-1}$ is well defined.
For convenience we introduce the notation $\alpha_t \defas \|\zz_t - \xx^\star\|^2 + (1 - \sigma)^{-1}\|\gamma_{t-1} (\uu_t - \uu^\star)\|^2$. With this, the previous inequality can be simplified to
\begin{align}
  2 \gamma_t(\mathcal{L}(\xx_{t+1}, \uu^\star) - &\mathcal{L}(\xx^\star, \uu_{t+1})) + \alpha_{t+1} \leq  (1 - \rho) \|\zz_{t} - \xx^\star\|^2 +  \|\gamma_{t-1}(\uu_{t} - \uu^\star)\|^2~.
\end{align}
By the definition of saddle point we have $\mathcal{L}(\xx_{t+1}, \uu^\star) - \mathcal{L}(\xx^\star, \uu_{t+1}) \geq 0$. Dropping this non-negative term gives
\begin{align}
  \alpha_{t+1} \leq  (1 - \rho) \|\zz_{t} - \xx^\star\|^2 +  \|\gamma_{t-1}(\uu_{t} - \uu^\star)\|^2~.
\end{align}
We now make a distinction of cases based on the relative magnitude of $\rho$ and $\sigma$.
\begin{itemize}
\item If $\rho \leq \sigma$ then $(1 - \rho) \geq (1 - \sigma) \implies 1 \leq (1 - \rho)(1 - \sigma)^{-1}$ and so from the previous equation we have
\begin{align}
  \alpha_{t+1} &\leq (1 - \rho) \|\zz_{t} - \xx^\star\|^2 +  \|\gamma_{t-1}(\uu_{t} - \uu^\star)\|^2\\
  &\leq (1 - \rho) \|\zz_{t} - \xx^\star\|^2 +  (1 - \rho) (1 - \sigma)^{-1}\|\gamma_{t-1}(\uu_{t} - \uu^\star)\|^2 \\
  &= (1 - \rho) \alpha_{t}
\end{align}
\item   Otherwise, if $\sigma < \rho$, then $(1 - \rho) < (1 - \sigma)$ and we have
\begin{align}
 \alpha_{t+1} &\leq (1 - \rho) \|\zz_{t} - \xx^\star\|^2 +  \|\gamma_{t-1}(\uu_{t} - \uu^\star)\|^2\\
 &\leq (1 - \sigma) \|\zz_{t} - \xx^\star\|^2 +  \|\gamma_{t-1}(\uu_{t} - \uu^\star)\|^2 \\
  &= (1 - \sigma) \alpha_{t}
\end{align}
\end{itemize}
Combining the two previous equations we have
\begin{equation}
  \alpha_{t+1} \leq \Big(1 - \min\{\rho,\sigma\}\Big)\alpha_{t}~,
\end{equation}
which leads by recurrence to
\begin{equation}\label{eq:basic_ht_recurrence2}
  \alpha_{t+1} \leq \Big(1 - \min\big\{\rho, \sigma\big\}\Big)^{t+1} \alpha_0~.
\end{equation}
By definition of $\alpha_{t+1}$ the previous inequality gives
\begin{equation}
     \|\zz_{t+1} - \xx^\star\|^2 + (1 - \sigma)^{-1}\|\gamma_{t} (\uu_{t+1} - \uu^\star)\|^2\leq \Big(1 - \min\big\{\rho, \sigma\big\}\Big)^{t+1} \alpha_0.
\end{equation}
From where we have the following geometric bounds for the primal and dual variables
\begin{align}
    \|\zz_{t+1} - \xx^\star\|^2 &\leq \Big(1 - \min\big\{\rho, \sigma\big\}\Big)^{t+1} \alpha_0\\
  \| \gamma_{t} (\uu_{t+1} - \uu^\star)\|^2 &\leq  (1 - \sigma)\Big(1 - \min\big\{\rho, \sigma\big\}\Big)^{t+1} \alpha_0\\
  &\leq  \Big(1 - \min\big\{\rho, \sigma\big\}\Big)^{t+2} \alpha_0\label{eq:recursive_bound_u}
\end{align}
This would be sufficient to derive a convergence rate in terms of $\|\zz_{t+1} - \xx^\star\|^2$. However, since our sublinear convergence result was in terms of $\xx_{t+1}$, we would like to state this result in terms of $\xx_{t+1}$ too. This is possible with a small amount of work and loosing a constant factor $6$.

Using the primal-dual relationship $\xx_{t+1} = \zz_{t+1} + \gamma_{t}(\uu_{t+1}-\uu_{t})$ and an improved version of Young's inequality (Lemma~\ref{lemma:young}) $\|\boldsymbol a+\boldsymbol b+ \boldsymbol c\|^2 \leq 3(\|\boldsymbol a\|^2+ \|\boldsymbol b\|^2 + \|\boldsymbol c\|^2)$ we have
\begin{align}
    \|\xx_{t+1}-\xx^\star\|^2
    &= \|\zz_{t+1}  + \gamma_{t}(\uu_{t+1}-\uu_*) + \gamma_t (\uu_*-\uu_{t}) -\xx^\star\|^2 \quad \text{ (adding and substracting $\uu^\star$)} \\
    &\leq 3\|\zz_{t+1}-\xx^\star\|^2  + 3\|\gamma_{t}(\uu_{t+1}-\uu_{*})\|^2 + 3\|\gamma_{t}(\uu_{*}-\uu_{t})\|^2 \quad\text{ (Young's inequality)} \\
    &\leq 3 \|\zz_{t+1}-\xx^\star\|^2  + 3\|\gamma_{t}(\uu_{t+1}-\uu^\star)\|^2 + 3\|\gamma_{t-1}(\uu_{t}-\uu^\star)\|^2 \\
    & \qquad \text{ ( $\gamma_{t} \leq \gamma_{t-1}$)} \nonumber\\
    &\leq 3\|\zz_{t+1}-\xx^\star\|^2  + 3(1 - \sigma)^{-1}\|\gamma_{t}(\uu_{t+1}-\uu^\star)\|^2 + 3\|\gamma_{t-1}(\uu_{t}-\uu^\star)\|^2 \\
    &\qquad \text{ ($1 \leq (1 - \sigma)^{-1}$)} \nonumber\\
    &\leq 3\Big(1 - \min\big\{\rho, \sigma\big\}\Big)^{t+1} \alpha_0 \,+\, 3\|\gamma_{t-1}(\uu_{t}-\uu^\star)\|^2 \\
    &\qquad \text{ (by Eq.\eqref{eq:basic_ht_recurrence2})} \nonumber\\
    & \leq 3\Big(1 - \min\big\{\rho, \sigma\big\}\Big)^{t+1} \alpha_0 \,+\, 3\Big(1 - \min\big\{\rho, \sigma\big\}\Big)^{t+1} \alpha_0 \quad \text{ (by Eq. \eqref{eq:recursive_bound_u})}  \\
    & \leq \Big(1 - \min\big\{\rho, \sigma\big\}\Big)^{t+1} 6 \alpha_0~.
\end{align}
The claimed rate then follows by definition of $\alpha_0$.

\textbf{Variant 2:}
As in Variant 1, by the duality between Lipschitz gradient and strong convexity, $h$ being $L_h$-smooth implies that $h^*$ is \mbox{$L_h^{-1}$-strongly} convex.
Using Lemma~\ref{lemma:key_recursive} with $\xx = \xx^\star$, $\uu = \uu^\star$ we obtain the following inequality for Variant 2:
\begin{align}
  2 \gamma_t(\mathcal{L}(\xx_{t+1}, \uu^\star) - &\mathcal{L}(\xx^\star, \uu_{t+1})) + \|\zz_{t+1} - \xx^\star\|^2 + (1 - \xi)^{-1}\|\gamma_{t-1} (\uu_t - \uu^\star)\|^2 + 2\gamma_t \delta_t\\
  &\qquad\qquad\leq (1 - \rho) \|\zz_{t} - \xx^\star\|^2 +  \|\gamma_{t-1}(\uu_{t} - \uu^\star)\|^2 +  \gamma_{t-1}\delta_{t-1}~.,
\end{align}
where we note that the assumption $0 < L_h$ implies $\xi < 1$ and so $(1 - \xi)^{-1}$ well defined.
Once again by the definition of saddle point we can drop the first non-negative term to get
\begin{equation}
    \widetilde\alpha_{t+1} \leq (1-\rho)\|\zz_t-\xx^\star\|^2 + \|\gamma_{t-1}(\uu_t-\uu^\star)\|^2 + \gamma_{t-1} \delta_{t-1}~,
\end{equation}
where we noted $\widetilde\alpha_t \defas \|\zz_t-\xx^\star\|^2 + (1-\xi)^{-1} \|\gamma_t(\uu_{t-1}-\uu^\star)\|^2 + 2 \gamma_{t-1} \delta_{t-1}$. Similarly as for the first variant, we can make a distinction of cases to prove the contraction:
\begin{itemize}
  \item If $\rho < \xi$, then $(1 - \rho) \implies (1 - \rho)(1 - \xi) \geq 1$ and we have
  \begin{align}
    \widetilde\alpha_{t+1} &\leq (1 - \rho) \|\zz_{t} - \xx^\star\|^2 +  (1 - \rho) (1 - \xi)^{-1}\|\gamma_{t-1}(\uu_{t} - \uu^\star)\|^2 + \frac{1}{2} (2 \gamma_{t-1} \delta_{t-1}) \\
    &= \left(1 - \min\big\{\rho, \mfrac{1}{2}\big\}\right) \widetilde\alpha_{t}
  \end{align}
  \item   If $\xi < \rho$, then $(1 - \rho) < (1 - \xi)$ and we have
  \begin{align}
   \widetilde\alpha_{t+1} &\leq (1 - \xi) \|\zz_{t} - \xx^\star\|^2 +  \|\gamma_{t-1}(\uu_{t} - \uu^\star)\|^2 + \frac{1}{2} (2 \gamma_{t-1} \delta_{t-1})\\
    &= \left(1 - \min\big\{\xi, \mfrac{1}{2}\big\}\right)  \widetilde\alpha_{t}
    \end{align}
\end{itemize}
Combining the previous two bounds, we have the following geometric decrease on the sequence $(\widetilde\alpha_t)$,
\begin{equation}
    \widetilde\alpha_{t+1} \leq \Big(1 - \min\big\{\rho,\mfrac{1}{2},\xi\big\}\Big) \widetilde\alpha_t~.
\end{equation}
% As we did for the first part, $\widetilde\alpha_{t+1}$
% leading to the claimed rates.
Then by definition of $\widetilde\alpha_{t+1}$ we have
\begin{equation}\label{eq:basic_ht_recurrence_v2}
     \|\zz_{t+1} - \xx^\star\|^2 + (1 - \xi)^{-1}\|\gamma_{t} (\uu_{t+1} - \uu^\star)\|^2 + 2 \gamma_{t}\delta_t \leq \Big(1 - \min\big\{\rho, \sigma, \mfrac{1}{2}\big\}\Big)^{t+1} \widetilde\alpha_0.
\end{equation}
From where we can derive the following geometric bounds for the primal and dual variables:
\begin{align}
    \|\zz_{t+1} - \xx^\star\|^2 &\leq \Big(1 - \min\big\{\rho, \xi, \mfrac{1}{2}\big\}\Big)^{t+1} \tilde\alpha_0\\
  (1 - \xi)^{-1}\| \gamma_{t} (\uu_{t+1} - \uu^\star)\|^2 + 2 \gamma_{t}\delta_t&\leq  \Big(1 - \min\big\{\rho, \xi, \mfrac{1}{2}\big\}\Big)^{t+1} \widetilde \alpha_0\label{eq:bound_dual_var2}
\end{align}
We will now make a distinction of cases on $\xi$ to derive a more convenient bound for this last inequality. If $\xi \leq \frac{1}{2} \implies (1 - \xi)^{-1}\leq \frac{1}{2}$ and so we have
\begin{align}
(1 - \xi)^{-1}\| \gamma_{t} (\uu_{t+1} - \uu^\star)\|^2 + (1 - \xi)^{-1}\gamma_{t}\delta_t &\leq
(1 - \xi)^{-1}\| \gamma_{t} (\uu_{t+1} - \uu^\star)\|^2 + 4(1 - \xi)^{-1}\gamma_{t}\delta_t\\
&\qquad \text{ (by non-negativity of $\delta_t$)}\nonumber\\
&\leq  (1 - \xi)^{-1}\| \gamma_{t} (\uu_{t+1} - \uu^\star)\|^2 + 2 \gamma_{t}\delta_t\\
&\leq \Big(1 - \min\big\{\rho, \xi, \mfrac{1}{2}\big\}\Big)^{t+1} \widetilde \alpha_0  \label{eq:bound_u_xi1}
\end{align}
Multiplying both sides by $(1 - \xi)$ and using $(1 - \xi) \leq (1 - \min\big\{\rho, \xi, \mfrac{1}{2}\big\})$ we have
\begin{equation}
\| \gamma_{t} (\uu_{t+1} - \uu^\star)\|^2 + \gamma_{t}\delta_t\leq \Big(1 - \min\big\{\rho, \xi, \mfrac{1}{2}\big\}\Big)^{t+2} \widetilde \alpha_0  ~.
\end{equation}
If on the other hand, $\xi\geq \frac{1}{2} \implies (1 - \frac{1}{2})(1 - \xi)^{-1}\leq 1$ and so we have
\begin{align}
    \| \gamma_{t} (\uu_{t+1} - \uu^\star)\|^2 + \gamma_{t}\delta_t &\leq (1-\mfrac{1}{2})\left( (1 - \xi)^{-1}\| \gamma_{t} (\uu_{t+1} - \uu^\star)\|^2 +  2\gamma_{t}\delta_t \right)\\
    &\leq  (1 - \mfrac{1}{2})\Big(1 - \min\big\{\rho, \xi, \mfrac{1}{2}\big\}\Big)^{t+1} \widetilde \alpha_0\\
    &\leq \Big(1 - \min\big\{\rho, \xi, \mfrac{1}{2}\big\}\Big)^{t+2} \widetilde \alpha_0\label{eq:bound_u_xi2}
\end{align}
By \eqref{eq:bound_u_xi1} and \eqref{eq:bound_u_xi2}, we see that in both cases we have the following bound:
\begin{equation}
  \| \gamma_{t} (\uu_{t+1} - \uu^\star)\|^2 + \gamma_{t}\delta_t\leq \Big(1 - \min\big\{\rho, \xi, \mfrac{1}{2}\big\}\Big)^{t+2} \widetilde\alpha_0~.\label{eq:recursive_bound_u_2}
\end{equation}

To obtain a convergence rate in terms of $\|\xx_{t+1} - \xx^\star\|^2$, we will use the primal-dual relationship ${\xx_{t+1} = \zz_{t+1} + \gamma_{t}(\uu_{t+1}-\uu_{t})}$ and an improved version of Young's inequality (Lemma~\ref{lemma:young}) $\|\boldsymbol a+\boldsymbol b+ \boldsymbol c\|^2 \leq 3(\|\boldsymbol a\|^2+ \|\boldsymbol b\|^2 + \|\boldsymbol c\|^2)$. We have the following sequence of inequalities:
\begin{align}
    \|\xx_{t+1}-\xx^\star\|^2
    &= \|\zz_{t+1}  + \gamma_{t}(\uu_{t+1}-\uu_*)+ \gamma_t (\uu_*-\uu_{t})) -\xx^\star\|^2\\
    &\qquad \text{ (adding and substracting $\uu^\star$)} \nonumber\\
    &\leq 3\|\zz_{t+1}-\xx^\star\|^2  + 3\|\gamma_{t}(\uu_{t+1}-\uu_{*})\|^2 + 3\|\gamma_{t}(\uu_{t}-\uu_{*})\|^2 \\
    &\qquad \text{ (Young's inequality)} \nonumber\\
    &\leq 3 \|\zz_{t+1}-\xx^\star\|^2  + 3\|\gamma_{t}(\uu_{t+1}-\uu^\star)\|^2 + 3\|\gamma_{t-1}(\uu_{t}-\uu^\star)\|^2 + 3(\gamma_{t-1} \delta_{t-1} (2 \beta_h)^{-2})\|\uu_t - \uu^\star\|^2 \nonumber \\
    &\qquad \text{ (by the maximum step-size increase $\gamma_{t}^2 \leq \gamma_{t-1}^2 + \gamma_{t-1} \delta_{t-1} (2 \beta_h)^{-2}$)}\\
    &\leq 3 \|\zz_{t+1}-\xx^\star\|^2  + 3\|\gamma_{t}(\uu_{t+1}-\uu^\star)\|^2 + 3\|\gamma_{t}(\uu_{t}-\uu^\star)\|^2 + 3 \gamma_{t-1} \delta_{t-1} \\
    &\qquad \text{ (by the Lipschitz assumption on $h$ have $\|\uu_t - \uu^\star\|^2\leq 4\beta_h^2$)}\\
    &\leq 3\Big(1 - \min\big\{\rho, \xi, \mfrac{1}{2}\big\}\Big)^{t+1} \widetilde  \alpha_0 \,+\,  3\|\gamma_{t}(\uu_{t}-\uu^\star)\|^2 + 3 \gamma_{t-1} \delta_{t-1} \\
    &\qquad \text{ (by Eq.\eqref{eq:basic_ht_recurrence_v2} and using the bound $(1 - \xi)^{-1} > 1$)} \nonumber\\
    & \leq 3\Big(1 - \min\big\{\rho, \xi, \mfrac{1}{2}\big\}\Big)^{t+1} \widetilde\alpha_0 \,+\, 3\Big(1 - \min\big\{\rho, \xi, \mfrac{1}{2}\big\}\Big)^{t+1} \widetilde \alpha_0 \\
    &\qquad \text{ (by Eq. \eqref{eq:recursive_bound_u_2})} \nonumber \\
    & \leq \Big(1 - \min\big\{\rho, \xi, \mfrac{1}{2}\big\}\Big)^{t+1} 6 \widetilde\alpha_0.
\end{align}
The desired rate is then a consequence of the definition of $\widetilde\alpha_0$, using that by definition $\delta_{-1} = 0$.
\end{proof}

\vspace{2em}

\subsection{Convergence rates comparison}\label{apx:comparison_convergence_rates}

We compare the obtained convergence rate against
those in \citet{davis2015three} for the fixed step-size strategy.
% We first note that the quantity $\tau$ only enters in the analysis to bound $\gamma_t$ by $1/L_f$, and so for the analysis of the fixed step-size
%
% We first note for a constant step-size $\gamma_t \leq 1/L_f$, the proofs can used to derive a rate with $\tau=1$ since this quantity only enters to bound $\gamma_t$ in the general case by $\gamma_t \leq \tau/L_f$.

\paragraph{Sublinear convergence.}
Our own Corollary~\ref{cor:sublinear_convergence} and \citep[Corollary D.5.2]{davis2015three} provide a convergence rate of the form
\begin{equation}
  P(\overline\xx_{t+1}) - P(\xx^\star) \leq \frac{L_f}{2 (t+1)} Q_0~,
\end{equation}
with a different definition of $Q_0$ in both cases:
\begin{align}
\text{This work: }Q_0 &\defas \|\zz_0 - \xx^\star\|^2 + \frac{2}{L_f^2} \left( \|\uu_0\|^2 + \beta_h^2\right)\\
  \text{\citep{davis2015three}: } Q_0 &\defas \|\yy_0 - \xx^\star\|^2 + \frac{1}{2}\|\yy_0 - \yy^*\|^2 + \frac{4}{L_f}\|\yy_0 - \yy^*\|\|\nabla f(\xx^\star)\| + \frac{4 \beta_h}{L_f}\|\yy_0 - \yy^*\|~,
\end{align}
where $\yy_0 \defas \zz_0 + \frac{1}{L_f} \uu_0$, $\yy^* \defas \xx^\star + \frac{1}{L} \uu^\star$ the value of $Q_0$ for \citep{davis2015three} was obtained by plugging the step-size $\gamma=1/L_f$ in their Corollary D.5.2, and optimizing with respect to their $\varepsilon$ parameter.

Both quantities are difficult to compare, but it is instructive to compare them on specific cases. If $h=0$ and $\uu_0=\uu^\star=0$ we have
\begin{align}
\text{This work: }Q_0 &= \|\zz_0 - \xx^\star\|^2 \\
  \text{\citep{davis2015three}: } Q_0 &= \frac{3}{2}\|\zz_0 - \xx^\star\|^2 + \frac{4}{L_f}\|\zz_0 - \xx^\star\|\|\nabla f(\xx^\star)\| + \frac{4 \beta_h}{L_f}\|\zz_0 - \xx^\star\|~,
\end{align}
And in this case it is clear that our bound is better. Consider now the case $f=g=0$ and $\zz_0 = \xx^\star = 0$. Then we have the following value of $Q_0$:
\begin{align}
\text{This work: }Q_0 &\defas\frac{2}{L_f^2} \left( \|\uu_0\|^2 + \beta_h^2\right)\\
  \text{\citep{davis2015three}: } Q_0 &\defas \frac{1}{L_f^2}\|\uu_0\|^2 + \frac{1}{2L_f^2}\|\uu_0 - \uu^\star\|^2  + \frac{4 \beta_h}{L_f^{3/2}}\|\uu_0 - \uu^\star\|~,
\end{align}
in this case however the bound of \citep{davis2015three} can be better as it avoids the quadratic depency on $\beta_h$, although the quantity $\|\uu_0 - \uu^\star\|$ is also in the worse case in the order of $2\beta_h$.

\paragraph{Linear convergence.}
For a step-size $\gamma=1/L_f$, the convergence rate of \citep[Theorem D.6, point 6]{davis2017three} gives a convergence of the form (using our notation and the step-size $\gamma=1/L_f$):
\begin{equation}
\frac{2 \mu_f (1 - \eta)}{L_f(1 + \gamma L_h)^2}
\end{equation}
where $\eta$ is a quantity that verifies the inequalities
\begin{equation}
  \eta > \frac{1}{\varepsilon}, \text{ with } \varepsilon > \frac{1}{2}
\end{equation}
by the choice of step-size (see their Theorem D.6) for a definition of these quantities. Optimizing with respect to $\eta$ gives $\eta < 1/2$ and so the rate can be arbitrarily close to
\begin{equation}
\frac{ \mu_f}{L_f(1 + \gamma L_h)^2}   = \rho \sigma^2~.
\end{equation}

This bound is clearly \emph{strictly worse} than our $\min\{\rho, \sigma\}$ of Theorem~\ref{thm:linear_convergence} for Variant 1 with $\tau=1$.
The constant factor $D_0$ is worse than the one obtained by~\citet{davis2015three}, but only by a factor $6(1 - \sigma)^{-1}$, and note also that the rate of \citet{davis2015three} is not given in terms of the primal variables but in terms of the less interpretable quantity $\|\xx_t + \gamma \uu_{t+1} - \xx^\star - \gamma \uu^\star \|^2$.

The difference between our convergence rate and that of \citet{davis2015three} can be quite large, as illustrated in the plot below for values in the interval $[0, \frac{1}{2}]^2$:

\includegraphics[width=0.5\linewidth]{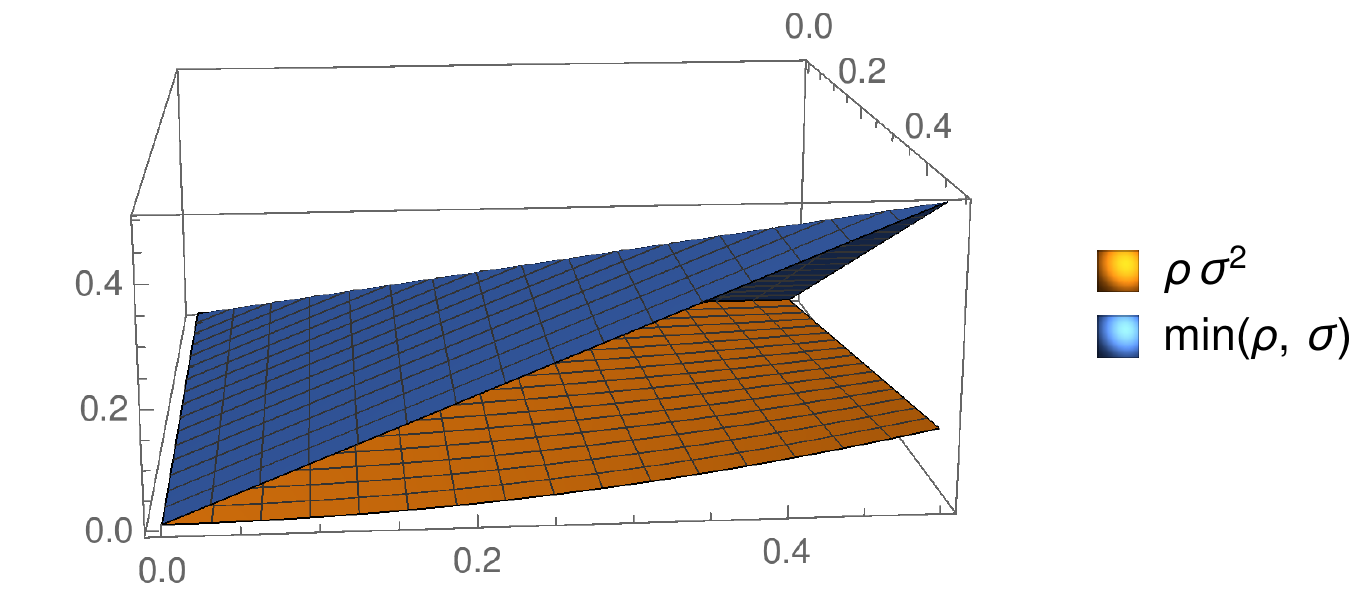}
\clearpage

\section{Proximal operators}\label{apx:proximal_operators}

In this section we provide a more detailed description of the different splitting methods.
We start with a reminder of some penalties that have a known closed form proximal operator.
\begin{itemize}
\item {\bfseries $\ell_1$ norm.} The proximal operator of the $\ell_1$ norm $\|\xx\|_1 \defas \sum_{j=1}^p|\xx_j|$ norm is given by the soft-thresholding operator. More precisely, for $\xx \in \RR^p$ it is given componentwise as
\begin{equation}\label{eq:prox_l1}
[\prox_{\gamma \|\cdot\|}(\xx)]_j = \begin{cases}\left(1 - \frac{\gamma}{|\xx_j|} \right) \xx_j~, &\text{ if $|\xx_j| \geq \gamma$}\\
0 &\text{otherwise}
\end{cases}
\end{equation}
for $j \in \{1, \ldots, p\}$.
\item {\bfseries Group lasso norm.} Let $\mathcal{G}$ be a partition of $\{1, \ldots, p\}$. Then the group lasso or $\ell_1/\ell_2$ norm for the $\mathcal{G}$ partition is defined as $\|\xx\|_{\mathcal{G}} \defas \sum_{g \in \mathcal{G}}\|[\xx]_g\|_2$. Its proximal operator is given as
\begin{equation}
[\prox_{\gamma \|\cdot\|_{\mathcal{G}}}(\xx)]_g =
\begin{cases}\left(1 - \frac{\gamma}{\|[\xx]_g\|} \right)_+ [\xx]_g~, &\text{ if $\|[\xx]_g\| \geq \gamma$}\\
0 &\text{otherwise}
\end{cases}
% \left(1 - \frac{\gamma}{\|[\xx]_g\|} \right)_+ [\xx]_g~, \text{ for $g \in \mathcal{G}$}~.
\end{equation}
A proof of this can be found in~\citep[Proposition 1]{yuan2006model}.
\end{itemize}

\subsection{$\ell_1$ trend filtering.}

This penalty can be split as the sum of three proximal terms $h_1, h_2$ and $h_3$ as follows:
\begin{align}
\|\xx\|_\text{TF} &\defas \textstyle\sum_{i=1}^{p-2} |\xx_i - 2 \xx_{i+1} + \xx_{i+2}|\\
&= \underbrace{\sum_{i=1}^{(p-2)/3} |\xx_i - 2 \xx_{i+1} + \xx_{i+2}|}_{= h_1(\xx)} + \underbrace{\sum_{i=1}^{(p-2)/3}|\xx_{i+1} - 2 \xx_{i+2} + \xx_{i+3}|}_{= h_2(\xx)}+ \underbrace{\sum_{i=1}^{(p-2)/3}|\xx_{i+2} - 2 \xx_{i+3} + \xx_{i+4}|}_{= h_3(\xx)}
\end{align}
The term $h_1$ can be written as $\|\boldsymbol{L} \xx\|_1$ for an $\boldsymbol{L}$ matrix of the form
\begin{equation}
\boldsymbol{L} = \begin{bmatrix}
    1 & -2 & 1 & 0 & 0  & 0 & \dots \\
    0 &  0 & 0 & 1 & -2 & 1 & \ddots & \\
    \vdots & \ddots & \\
    0 &  0 & 0 & 0 & 0 & 0 & 1 & -2 & 1 \\
    \end{bmatrix}
\end{equation}
Usually, penalties of the form $\|\boldsymbol{L} \xx\|_1$ do not have a closed form proximal operator. However, this matrix $\boldsymbol{L}$ has a very special property that will allow to compute its proximal operator in closed form. In particular, this matrix is a \emph{semi-orthogonal} matrix, i.e., it verifies $\boldsymbol{L}\boldsymbol{L}^T = \nu \boldsymbol{I}$ for some $\nu > 0$ ($\nu=6$ in our case). For these matrices, the proximal operator of $\varphi(\boldsymbol{L}\xx)$ can be computed as (see e.g.~\citep[Table 1]{combettes2011proximal}):
\begin{equation}
\prox_{\varphi(\boldsymbol{L}\cdot)} = \xx + \nu^{-1}\boldsymbol{L}^T(\prox_{\nu \varphi}(\boldsymbol{L}\xx) - \boldsymbol{L}\xx),
\end{equation}
where in this case $\varphi$ is the $\ell_1$ norm and so its proximal operator is the soft thresholding operator of Eq.~\eqref{eq:prox_l1}.

\subsection{Isotonic and nearly-isotonic penalties}\label{apx:isotonic}

The isotonic constraint $\{\xx_1 \leq \xx_2 \leq \ldots \leq \xx_p\}$ can be enforced through the use of the indicator function ${\imath\{\xx_1 \leq \xx_2 \leq \xx_3 \leq \xx_4 \leq  \cdots\ \leq \xx_p\}}$.
Suppose that $p=2$, i.e., we only have the constraint $\{\xx_1 \leq \xx_2\}$. In this case, the projection of $(\zz_1, \zz_2)$ can be computed as follows: if $\zz_1 \leq \zz_2$ the projection does obviously nothing, otherwise it can be computed as the projection onto the line generated by the vector $(1, 1)$. This gives:
\begin{equation}\label{eq:prox_2_monotonic}
\prox_{\imath\{\xx_1 \leq \xx_2\}}(\xx_1, \xx_2) =
\begin{cases}(\xx_1, \xx_2) & \text{ if $\xx_1 \leq \xx_2$} \\
((\xx_1 + \xx_2) /  {2}, ({\xx_1 + \xx_2}) / {2}) & \text{ otherwise.})
\end{cases}
\end{equation}

Now lets consider the general case. The full indicator function can be decomposed into a sum of the two following terms
\begin{equation}
  \imath\{\xx_1 \leq \xx_2 \leq \xx_3 \leq \xx_4 \leq  \cdots\} = \underbrace{\imath\{\xx_1\! \leq \!\xx_2; \xx_3 \leq \xx_4;  \cdots\}}_{=g(\xx)}  + \underbrace{\imath\{\xx_2 \leq \xx_3; \xx_4 \leq \xx_5; \cdots\}}_{=h(\xx)}~.
\end{equation}
where the both terms are block separable with blocks of size 2. We can hence use the proximal operator of Eq.~\eqref{eq:prox_2_monotonic} to derive the proximal operator of $g$ and $h$, which by block-separability is merely a concatenation of the previous one. These proximal operators are given block-wise as
\begin{align}
[\prox_{\gamma g}(\zz)]_{\{2 i, 2 i + 1\}} &=
\begin{cases}(\xx_{2i}, \xx_{2i+1}) & \text{ if $\xx_{2i} \leq \xx_{2i+1}$} \\
((\xx_{2i} + \xx_{2i+1}) /  {2}, ({\xx_{2i} + \xx_{2i+1}}) / {2}) & \text{ otherwise.})
\end{cases}\\
[\prox_{\gamma h}(\zz)]_{\{2 i + 1, 2 i + 2\}} &=
\begin{cases}(\xx_{2i+1}, \xx_{2i+2}) & \text{ if $\xx_{2i+1} \leq \xx_{2i+2}$} \\
((\xx_{2i+1} + \xx_{2i+2}) /  {2}, ({\xx_{2i+1} + \xx_{2i+2}}) / {2}) & \text{ otherwise.})
\end{cases}
\end{align}

\paragraph{The nearly isotonic penalty.} The nearly isotonic penalty was proposed by~\citet{tibshirani2011nearly} as a relaxation of the aforementioned isotonic constraints and is defined as
\begin{equation}
  \|\xx\|_{\text{iso}} \defas \sum_{i=1}^{p-1} \max\{\xx_{i}- \xx_{i+1}, 0\}~.
\end{equation}
The penalty is zero if $\xx_{i} \leq \xx_{i+1}$ and so it encourages the coefficients to be non-decreasing. It can be split and similarly as we did for the isotonic constraints:
\begin{equation}
  \sum_{i=1}^{p-1} \max\{\xx_{i}- \xx_{i+1}, 0\} = \underbrace{\sum_{i=1}^{\lfloor
  p/2\rfloor} \max\{\xx_{2i}- \xx_{2i+1}, 0\}}_{=g(\xx)} + \underbrace{\sum_{i=1}^{\lfloor
  (p-1)/2\rfloor} \max\{\xx_{2i+1}- \xx_{2i+2}, 0\}}_{=h(\xx)}~.
\end{equation}
Let $\psi(\xx_1, \xx_2) \defas \max\{\xx_1 - \xx_2, 0\}$. If we denote by $\partial_i$ the subgradient with respect to the $i$-th coordinate we have
\begin{align}
  \partial_1 \psi(\xx_1, \xx_2) = \begin{cases}
  0 & \text{ if } \xx_1 < \xx_2\\
  [0, 1] & \text{ if } \xx_1 = \xx_2\\
  1 & \text{ if } \xx_1 > \xx_2
  \end{cases} \quad\text{ and } \quad
  \partial_2 \psi(\xx_1, \xx_2) = \begin{cases}
  0 & \text{ if } \xx_1 < \xx_2\\
  [-1, 0] & \text{ if } \xx_1 = \xx_2\\
  -1 & \text{ if } \xx_1 > \xx_2
  \end{cases}
\end{align}

 then it is easy to verify using the subdifferential inclusion of Lemma~\ref{lemma:prox_characterization} that its proximal operator is given by
\begin{equation}
  \prox_{\gamma \psi}(\xx_1, \xx_2) = \begin{cases}(\xx_1, \xx_2) & \text{ if $\xx_1 \leq \xx_2$}\\
  (\xx_1 - \gamma, \xx_2 + \gamma) &\text{ if $\xx_1 - \gamma \geq \xx_2 + \gamma$}\\
  ((\xx_1 + \xx_2) / 2, (\xx_1 + \xx_2)/2) &\text{ otherwise }.
  \end{cases}
\end{equation}
The proximal operator of the nearly isotonic penalty can then be computed by applying block-wise the above proximal operator.

\subsection{Doubly stochastic constraints}

The set of double stochastic matrices is composed of square matrices with nonnegative entries, each of whose rows and columns sum to $1$, i.e., $\{\XX^T \boldsymbol{1} = \boldsymbol{1}, \boldsymbol{1}^T \XX = \boldsymbol{1}, \XX \geq \boldsymbol{0}\}$. The indicator function over this set can be split as
\begin{equation}
\underbrace{\imath\{\XX^T \boldsymbol{1} = \boldsymbol{1}, \boldsymbol{1}^T \XX = \boldsymbol{1}\}}_{=g(\XX)} + \underbrace{\imath\{\XX \geq \boldsymbol{0}\}}_{=h(\XX)} \quad.
\end{equation}
For any $\XX \in \RR^{p \times p}$, the proximal operator of $g$ and $h$ is then given in closed form as
\begin{align}
\prox_{\gamma g}(\XX) &= \XX + \left( \frac{1}{n}\boldsymbol{I} + \frac{\boldsymbol{1}^T \XX \boldsymbol{1}}{n^2} \boldsymbol{I} - \frac{1}{n}\XX\right)\boldsymbol{1}\boldsymbol{1}^T - \frac{1}{n}\boldsymbol{1}\boldsymbol{1}^T \XX\\
\prox_{\gamma h}(\XX) &= (\XX + |\XX|) / 2~,
\end{align}
where $|\cdot|$ takes the absolute value componentwise and $\boldsymbol{1}$ denotes the $p$-dimensional vector of ones.
The proof of this result can be found in \citep{lu2016fast}.

\clearpage

\section{Benchmarks}\label{apx:benchmarks}

In this appendix we provide a more detailed description of the experiments, as well as present an extended set of benchmarks.

\begin{itemize}
  \item \ref{apx:implementation} discusses the implementation of the different algorithms.
  \item  \ref{apx:synthetic} details the synthetic data generation process that we used in the synthetic benchmarks.
  \item Each of the experiments we run has a subsection in this appendix with an extended discussion and set of results. These are organized in the order they appear in Figure~\ref{fig:main_figure}: Overlapping group lasso (\ref{apx:group_lasso}), total variation (\ref{apx:deblurring}), sparse and low rank matrix recovery with trace norm $\ell_1$ penalty (\ref{apx:sparse_lowrank}), and nearly isotonic penalty (\ref{apx:nearly_isotonic_bench}).
\end{itemize}

\subsection{Implementation details}\label{apx:implementation}

Particular care has been taken to ensure a fair and extensive comparison.
We have implemented all algorithms in Python, and used the just-in-time compiler Numba for some non-vectorizable computationally demanding parts such as computing the fused lasso proximal operator in the total variation-regularized problems. We profiled and optimized all algorithms equally, and in some cases made modifications that depart from the canonical description when this improved significantly performance (see TOS-AOLS below).

All methods have been implemented in Python and except TOS-AOLS are part of the C-OPT optimization package~\citep{fabian_pedregosa_2017_437991}.

\paragraph{PDHG.}
The PDHG or Condat-V\~u algorithm~\citep{condat2013primal, vu2013splitting} depends on two step-size parameters $\tau$ and $\sigma$. We parametrize the second step-size as $\sigma=\beta/\tau$
(this gave a better results than the parametrization $\sigma=\beta\tau$ used by \citet{malitsky2016first} for the range $\beta$ considered)
 then by~\citep[Theorem 3.1]{condat2013primal} $\tau$ needs to be of the form
\begin{equation}
  \tau \leq \frac{2 (1 - \beta)}{L_f}~,~\text{ with $\beta < 1$ }.
\end{equation}
With the above value of $\tau$, we tested three different values for $\beta$, $\beta=0.1, 0.5, 0.9$, and selected the one that had overall better performance (we found this to be $\beta=0.5$) hence giving it a slight advantage with respect to other methods that do not require step-size tuning.

\paragraph{Adaptive PDHG.}
Although one of the step-sizes is computed by a line search method, due to the two step-sizes it still requires to select the parametrization constant $\beta$. This parameter was computed the same way than for PDHG, except in this case we found the best performing value was $\beta=0.1$.
In this method the step-size is allowed to grow, and the initial value for the step-size in next iterate needs to be selected from an interval of admissible values.
 Following the experimental section of~\citep{malitsky2016first}, we always take the maximum admissible step-size.

As it is visible from some of the benchmarks (see e.g., Figures \ref{fig:group_lasso_synthetic}, \ref{fig:trace_norm}, \ref{fig:isotonic}), the method seems to suffer more than other methods from numerical instabilities in the high precision regime (typically when the suboptimality has reached $10^{-10}$).  We investigated this issue and developed an implementation that uses Moreau's decomposition to replaces $\prox_{ \tau h^*}$ by $\prox_{h/\tau}$, thinking that perhaps the different scales between $\tau$ and $\beta/\tau$ would be the cause of the instability, but this did not solve the issue. Interestingly, these instabilities did not appear on the non adaptive variant, despite both algorithms share much of the same code.

\paragraph{TOS-AOLS.} We implemented the averaged operator line search method  of \citet{giselsson2016line}, using TOS as the averaged operator and name it TOS-AOLS. Our implementation has a crucial improvement with respect to the version detailed in the original reference that we found was crucial to obtain competitive results. In our implementation, the line search multiplicative factor $\alpha_t$ (in their notation) is not bound to start from an upper bound $\alpha_{\text{max}}$ as is described in the reference but is instead incremented as $\alpha_{t+1} = 1.05 \alpha_t$ in case of success of the line search condition. This gave much better empirical results.

\paragraph{Adaptive TOS.} As described in \S\ref{scs:methods}.

\subsection{Synthetic data generation}\label{apx:synthetic}

For the experiments with synthetic data (except for the image deblurring task), we followed the same data generation process as in \citep{agarwal2010fast}, which generates a design matrix $\{\boldsymbol{a}_i\}_{i=1}^n$ according to the recursion
\begin{align}
\boldsymbol{a}_1 &= \zz_1\\
\boldsymbol{a}_i &= \zz_i + p \boldsymbol{a}_{j-1} \text{ for $i> 1$}
\end{align}
with $\zz_i$ sampled from a standardized Gaussian distribution. $p \in [0, 1)$ is a correlation parameter that makes the problem more ill-conditioned as $p \to 1$. This is necessary as otherwise the resulting dataset has an unrealistic perfect spectrum. In the experiments we chose $p=0.95$, which for a  matrix of size $65 \times 65$ gives a condition number of around 60. Although this value of $p$ might seem high, the resulting condition number is in fact smaller than the one found in the real datasets. For example, the \texttt{real-sim} dataset has a condition over 200.

% \vspace{2em}

\subsection{Overlapping group lasso benchmarks}\label{apx:group_lasso}

We consider an overlapping group lasso penalty with the following groups of size 10:
$\{\{1, \ldots, 10\}, \{8, \ldots, 18\}, \{16, \ldots, 26\}, \ldots\}$, where each group has an overlap of 2 coefficient with the previous groups.
This is modeled after the synthetic experiments of in~\citep[\S9.1]{jacob2009group}.
This set of groups can be split into two set $\mathcal{G}$ and $\mathcal{H}$ of disjoint groups by including in $\mathcal{G}$ the odd groups and in $\mathcal{H}$ the even ones, i.e., $\mathcal{G} = \{\{\{1, \ldots, 10\}, \{16, \ldots, 26\},\ldots\}$, $\mathcal{H} = \{\{8, \ldots, 18\},\{24, \ldots, 34\}, \ldots\}$.

With this notation, we can write the overlapping group lasso-penalized logistic regression problem as follows,
\begin{equation}\label{eq:overlapping_group_lasso}
  \minimize_{\xx\in \RR^p}\underbrace{\frac{1}{n}\sum_{i=1}^n \log(1 + \exp(- b_i \boldsymbol{a}_i^T \xx))}_{=f(\xx)} + \underbrace{\lambda\sum_{G \in \mathcal{G}}\|[\xx]_G\|}_{=g(\xx)} +  \underbrace{\lambda\sum_{H \in \mathcal{H}}\|[\xx]_H\|}_{=h(\xx)}~,
\end{equation}
where $\lambda$ is a regularization parameter and $\{b_i, \boldsymbol{a}_i\}_{i=1}^n$ is the dataset that we will discuss later.

\paragraph{Lipschitz constant of the proximal term.} For any set of disjoint groups $\mathcal{G}$ and any vector $\xx$ we have ${\sum_{G \in \mathcal{G}}\|\xx\| \leq \sqrt{|\mathcal{G}|}\|\xx\|}$, where $|\mathcal{G}|$ denotes the cardinality of $\mathcal{G}$. Hence, the Lipschitz constant of $h$ is upper bounded by $\lambda\sqrt{|\mathcal{G}|}$. This is the value we used in the experiments.

\paragraph{Dasets.} We consider the following three datasets:
\begin{itemize}
  \item The \texttt{real-sim} dataset, retrieved from the libsvm dataset collection.\footnote{\url{https://www.csie.ntu.edu.tw/~cjlin/libsvmtools/datasets/binary.html}} The size of this dataset is $n =72,309, p = 20,958$ and a density of 2\% (i.e., 98\% of zero coefficients in the data matrix).

  \item The \texttt{RCV1} dataset~\citep{lewis2004rcv1}, also retrieved from the \texttt{libsvm} dataset collection. The dimensions in this case are $n=20,242$, $p=677,399$, with a density of 0.1\%.

  \item A synthetic dataset, in which $\boldsymbol{a}_i$ is generated according to \ref{apx:synthetic} and a vector of ``ground truth'' is generated by randomly selecting 10 groups and setting them to a random value generated from a standard Gaussian distribution (in this case the data matrix is fully dense).
  In Python, for a problem of size $n=100$, $p=1002$:
  \begin{python}
  import numpy as np
  np.random.seed(0)

  n_samples, n_features = 100, 1002
  groups = [np.arange(8 * i, 8 * i + 10) for i in range(125)]

  ground_truth = np.zeros(n_features)
  g = np.random.randint(0, len(groups), 10)
  for i in g:
      ground_truth[groups[i]] = np.random.randn()
  \end{python}
\end{itemize}

\paragraph{Benchmarks.} The results for this model and the above datasets can be seen in Figure~\ref{fig:group_lasso_real} for the \texttt{real-sim} and \texttt{RCV1} datasets and in Figure~\ref{fig:group_lasso_synthetic} for the synthetic dataset. On 11 out of 12 cases the Adaptive TOS (variant 2) algorithm is the best performing method, and only marginally slower in the other case.

\begin{figure}[h]
  \includegraphics[width=\linewidth]{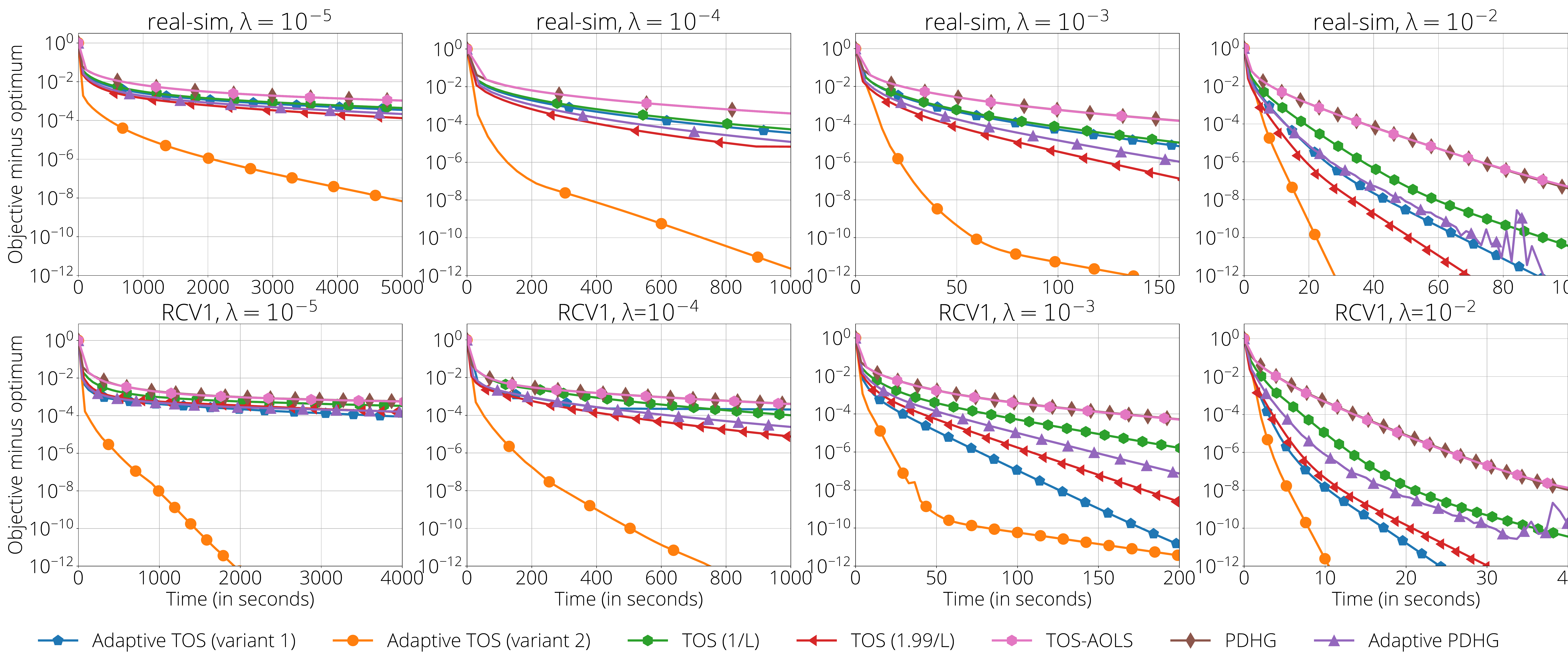}
  \caption{{\bfseries Logistic regression with overlapping group lasso penalty, text datasets}.  Top row: \texttt{real-sim} dataset, bottom row: \texttt{RCV1} dataset. The columns denote the amount of regularization, from a parameter $\lambda$ giving $\approx 50\%$ of zero coefficients (left) to a parameter giving $\approx 5\%$ of zero coefficients}\label{fig:group_lasso_real}
\end{figure}

\begin{figure}[h]
\includegraphics[width=\linewidth]{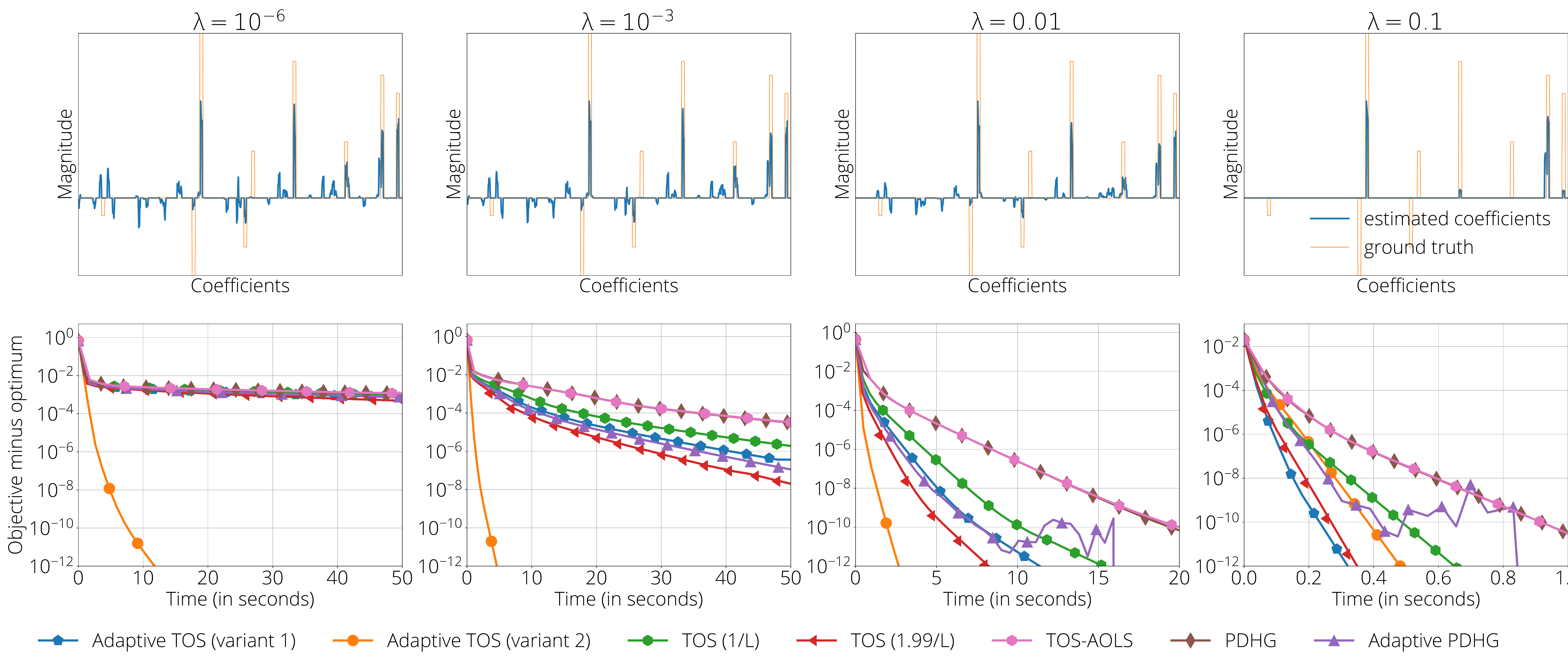}
\caption{{\bfseries Logistic regression with overlapping group lasso penalty, simulation dataset}. Top row: ground truth and estimated coefficients. Bottom row: time vs suboptimality comparison. Columns represent different values of the $\lambda$ regularization parameter}\label{fig:group_lasso_synthetic}
\end{figure}

\clearpage

\subsection{Total variation benchmarks (image deblurring with known blur operator)}\label{apx:deblurring}

In this benchmark we consider a classical image deblurring task with known blur operator. For this we choose a natural image shown below as ``Original image'' and generated another image by convolving the original image with a blur kernel and adding standardized Gaussian noise, shown below as ``Observed image''. The blur kernel is displayed in bottom left corner of Observed image
\begin{figure}[h]
\centering\includegraphics[width=0.6\linewidth]{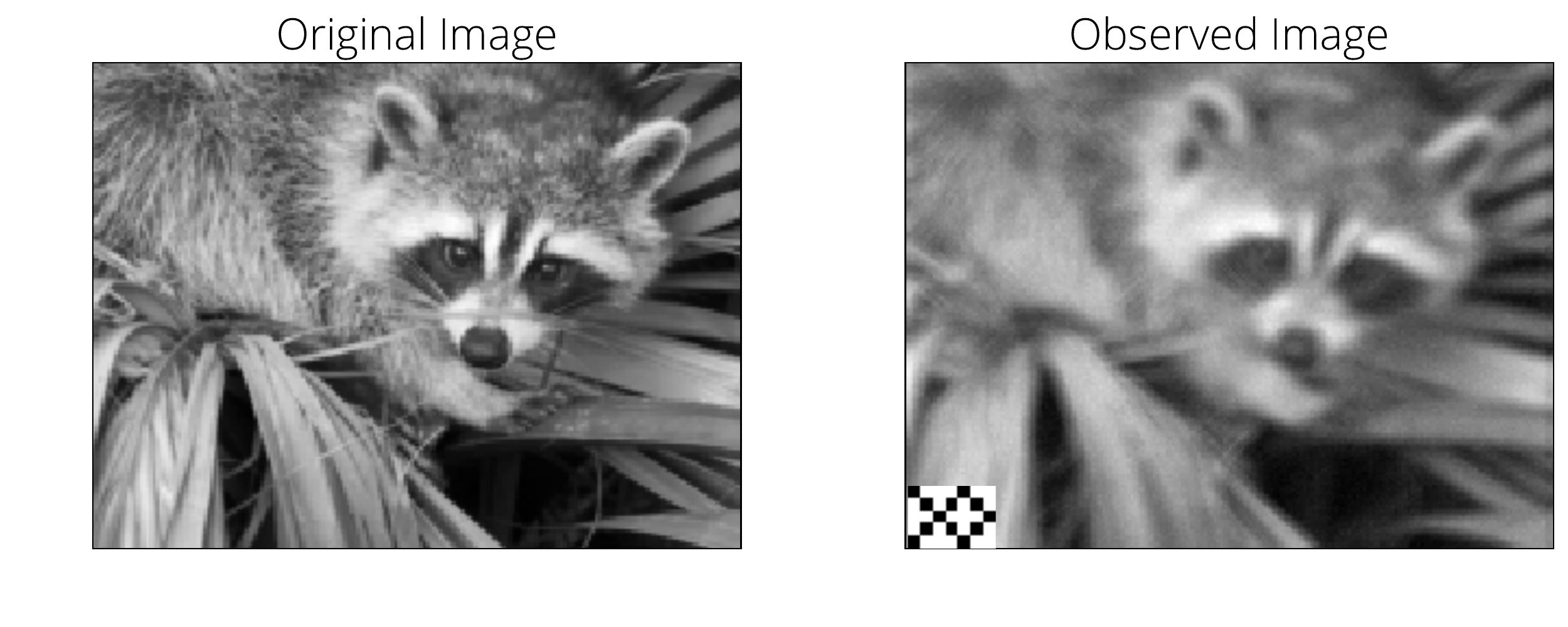}
\caption{{\bfseries Left}: Original image. {\bfseries Right}: Observed image, obtained by convolving with the kernel shown in lower left corner and with added Gaussian noise. Columns represent different values of the $\lambda$ regularization parameter}
\end{figure}

The image recovery can then be posed as least squares problem with a total variation penalty~\citep{rudin1992nonlinear, chambolle2016introduction} of the form:
\begin{align}
  &\minimize_{\XX \in \RR^{p \times q}} \|\boldsymbol{Y} - \boldsymbol{B} \XX \|^2 + \lambda\|\XX\|_{\text{TV}}~, \\
  \equiv&\minimize_{\XX \in \RR^{p \times q}} \underbrace{\|\boldsymbol{Y} - \boldsymbol{B} \XX \|^2\vphantom{\sum_i}}_{=f(\XX)}
   + \lambda\underbrace{\sum_{i=1}^p {\sum_{j=1}^{q-1}} |{\boldsymbol X}_{i, j+1} - {\boldsymbol X}_{i, j}|}_{ = g(\XX)}
   + \sum_{j=1}^q \underbrace{{\sum_{j=1}^{p-1}} |{\boldsymbol X}_{i+1, j} - {\boldsymbol X}_{i, j}|}_{ = h(\XX)}\,,\label{eq:tv_splitting}
\end{align}
where $\YY$ is the observed image, $\boldsymbol{B}$ is the blurring (linear) operator and $g$ and $h$ are fused lasso (also known as 1-dimensional total variation) penalties acting on the columns and rows of $\XX$ respectively. Their proximal operator was implemented using the algorithm of~\citet{condat2013direct}.

Alternatively, PDHG can solve this problem using a different splitting in which $g(\xx)=0$ and $\lambda\|\XX\|_{\text{TV}} = h(\KK \xx)$, where $h$ is the $\ell_1$ penalty and $\KK$ is the matrix of first order finite differences. We compared both splittings and concluded that the splitting of Eq.~\eqref{eq:tv_splitting} gave better results.

\paragraph{Lipschitz constant of the proximal term.}
For a matrix $\XX \in \RR^{p \times q}$ we have
\begin{equation}
\|\XX\|_{\text{FL}} = \|\boldsymbol{K}\text{vec}(\xx)\|_1 \leq \|\boldsymbol{K}\|_1 \|\text{vec}(\xx)\|_1 \leq 2 \|\text{vec}(\xx)\|_1\leq 2 \sqrt{p q}\|\text{vec}(\xx)\|~,
\end{equation}
and so the Lipschitz constant of $h$ can be bounded by $2 \sqrt{p q}$. This is the quantity we used in the experiments.

\paragraph{Benchmarks.} We show recovered image as well as benchmarks in Figure~\ref{fig:faces}. In this case, we found that the fixed step-size strategy with step-size $1.99/L$ marginally outperformed the Adaptive TOS on 3 out of 4 cases (but never by more than a factor of 1.5). Interestingly, this same method performed the worse on the high regularization setting (right column).

\begin{figure*}[h]
\includegraphics[width=\linewidth]{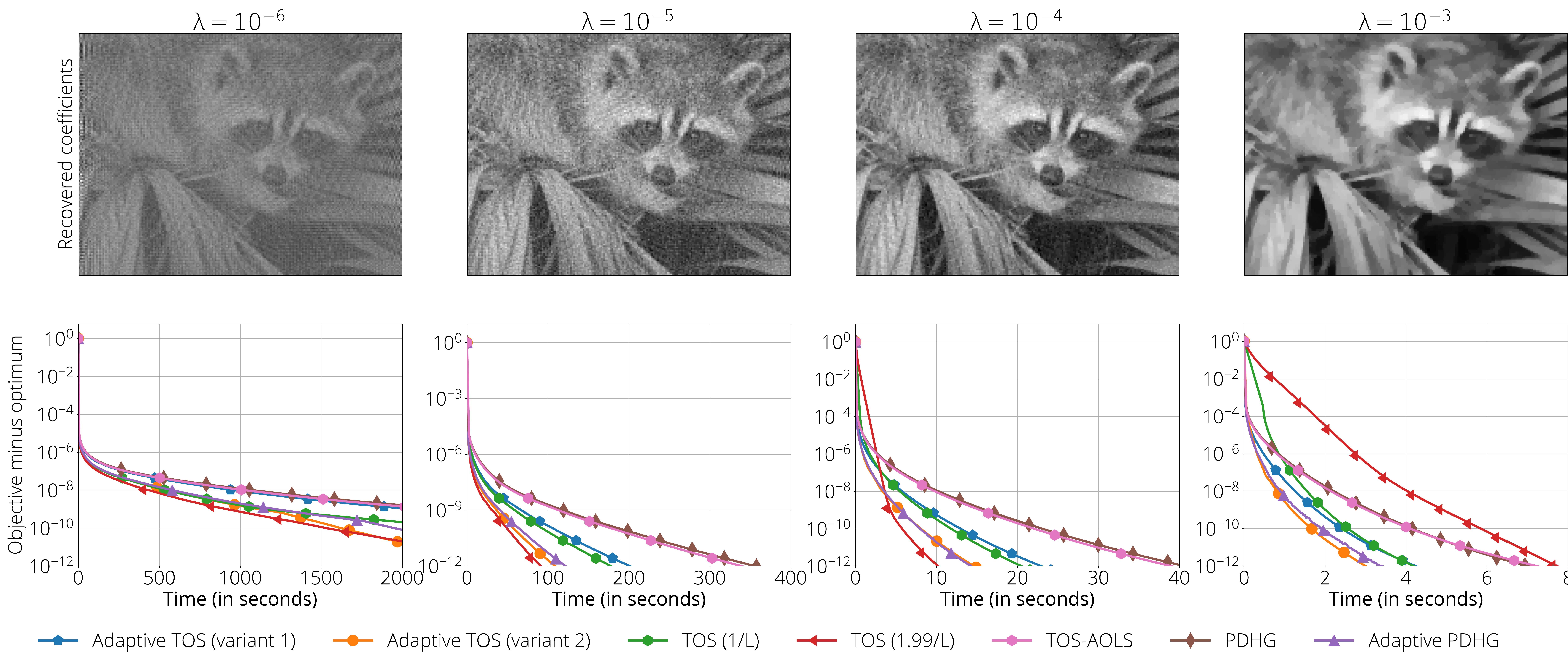}
\caption{{\bfseries Least Squares regression with total variation penalty}. Top row: recovered image. Bottom row: time vs suboptimality comparison. Columns represent different values of the $\lambda$ regularization parameter}\label{fig:faces}
\end{figure*}

\clearpage

\subsection{Sparse and low rank matrix recovery benchmarks}\label{apx:sparse_lowrank}

\begin{minipage}[t]{.7\linewidth}
We generated a sparse and low rank $20 \times 20$ symmetric matrix as the one in the right using the procedure detailed in~\citep{richard2012estimation}. This matrix can be seen in the right. We will denote this matrix by $\XX_{\text{truth}}$.
We then generate the target values as $b_i$ as $b_i = \boldsymbol{a}_i^T \text{vec}(\XX_{\text{truth}}) + \varepsilon_i$,
where $\varepsilon$ is a random noise, generated from a zero-mean, unit variance Gaussian distribution. We generated 200 of these samples, so that the problem has twice as many samples as features.

It has been shown that it is possible to promote sparse and low rank solutions by using a penalty composed of a trace (or nuclear) norm and a (vector) $\ell_1$ norm (see e.g., \citep{richard2012estimation}), where the trace norm is the sum of the  absolute values and we will denote it $\|\cdot\|_*$.
\end{minipage}
\begin{minipage}[t]{.28\linewidth}\vspace{-0.5em}
\includegraphics[width=\textwidth]{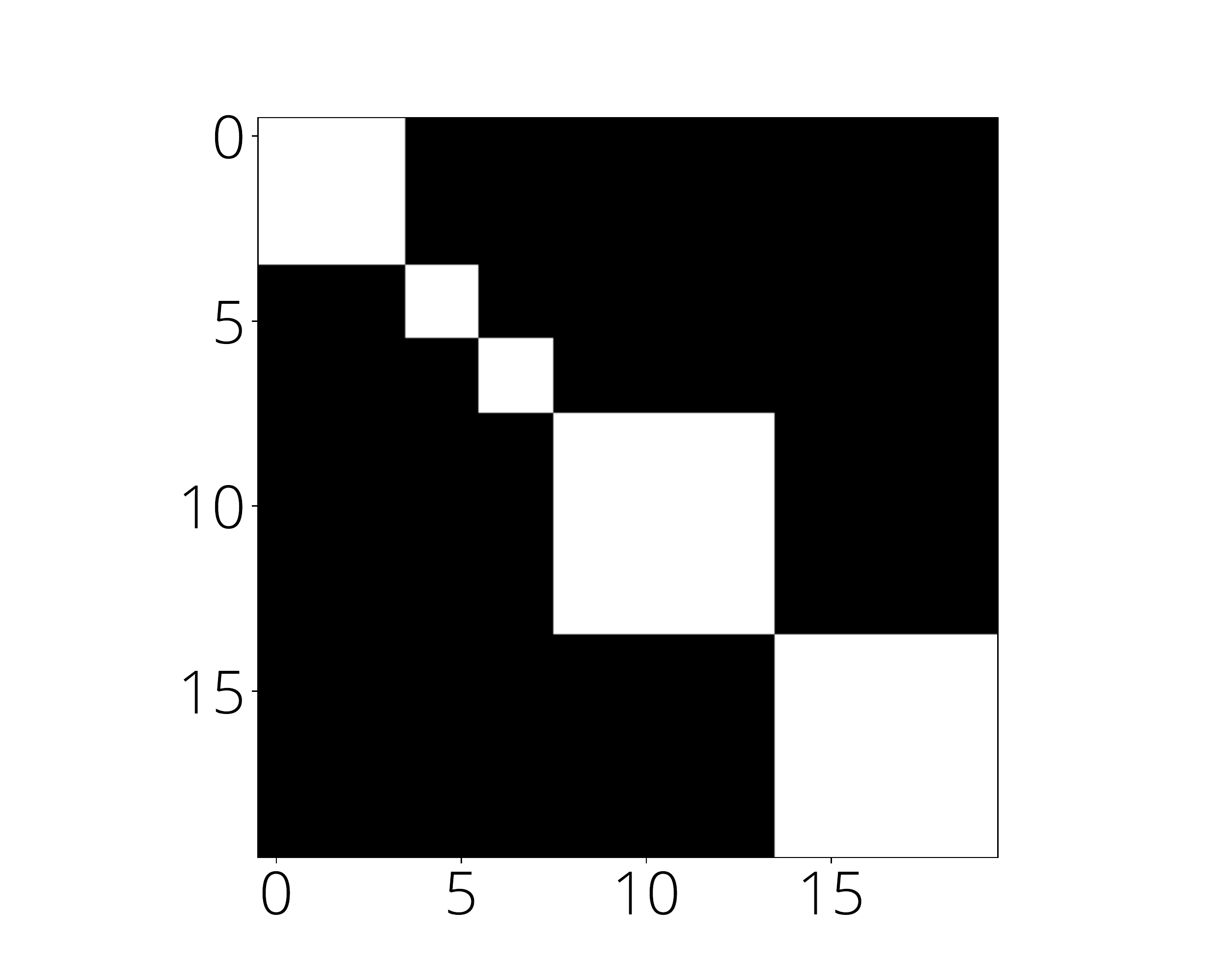}
  % \caption{Sparse block-diagonal matrix used as ground truth for this problem.}
\end{minipage}

In all, we consider the following problem:
\begin{equation}\label{eq:trace_norm}
  \minimize_{\XX\in \RR^{20\times 20}} \underbrace{\frac{1}{n}\sum_{i=1}^n (b_i - \boldsymbol{a}_i^T \text{vec}(\XX))^2}_{=f(\XX)} + \underbrace{\lambda \|\XX\|_*\vphantom{\sum_i}}_{=g(\XX)} + \underbrace{\vphantom{\sum_i}\mu \|\text{vec}(\XX)\|_1}_{=h(\XX)}~,
\end{equation}
where the proximal operator of $g$ is given by soft thresholding the singular values~\citep{cai2010singular}.

\paragraph{Lipschitz constant of proximal term.} By the properties of vector norms, for a matrix $\XX$ of size $p\times q$ we have that $\|\text{vec}(\XX)\|_1 \leq \sqrt{pq}\|\text{vec}(\XX)\|$. Hence, we can bound the Lipschitz constant of $h$ by  $\lambda\sqrt{pq}$.

\paragraph{Benchmarks.}
The results of the time comparison and recovered coefficients can be see in the Figure~\ref{fig:trace_norm} below. For simplicity we consider $\lambda=\mu$, but of course more accurate coefficients could be recovered by tuning both regularization parameters.

\begin{figure}[h]
\includegraphics[width=\linewidth]{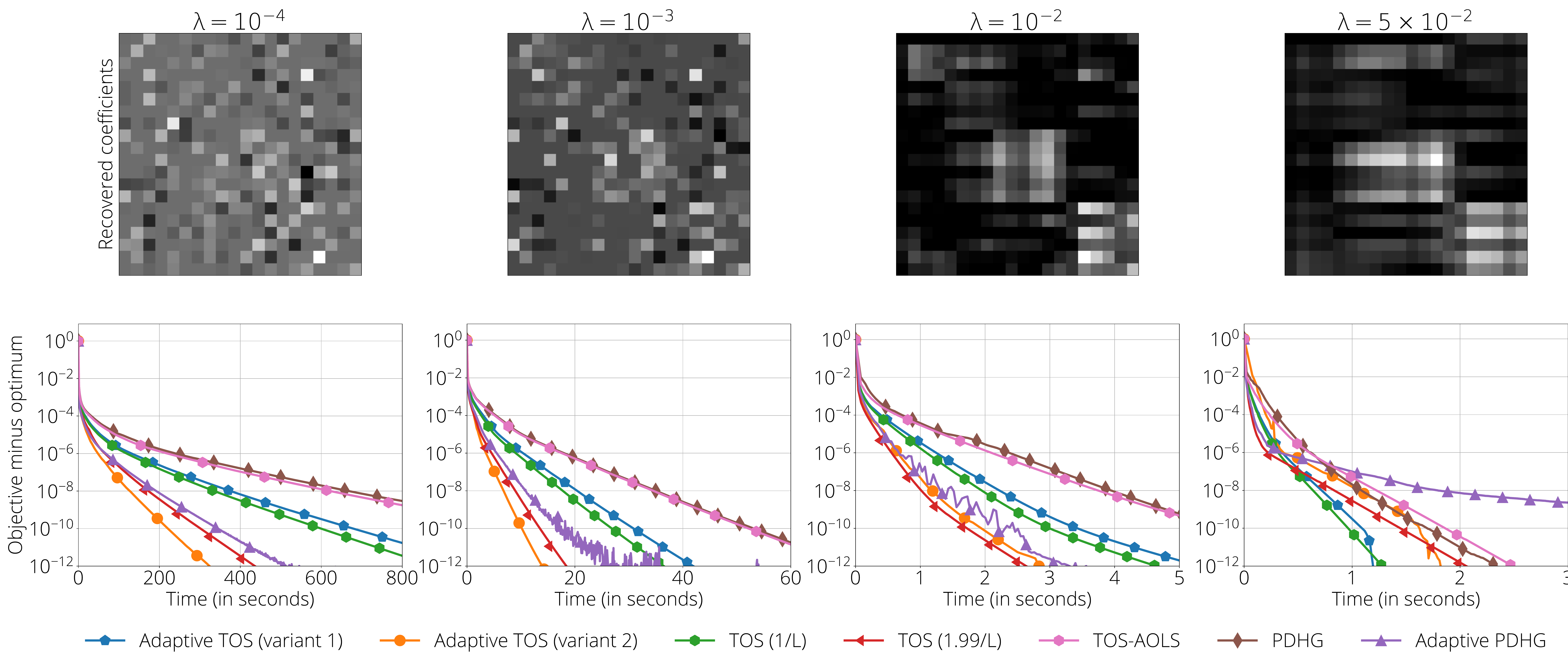}
\caption{{\bfseries Least Squares with trace norm + $\ell_1$ penalty}. Top row: recovered coefficients from problem \eqref{eq:trace_norm}. Bottom row: time vs suboptimality comparison. Columns represent different values of the $\lambda$ regularization parameter.}\label{fig:trace_norm}
\end{figure}

\clearpage

\subsection{Nearly isotonic penalty benchmarks}\label{apx:nearly_isotonic_bench}

Given an input dataset $\{(b_i, \boldsymbol{a}_i)\}_{i=1}^n$, which we generated according to  \ref{apx:synthetic}, we consider a logistic regression problem with nearly isotonic penalty on the weights:
\begin{align}\label{eq:logistic_nearly_isotonic}
  &\minimize_{\xx\in \RR^p}\frac{1}{n}\sum_{i=1}^n \log(1 + \exp(- b_i \boldsymbol{a}_i^T \xx)) + \lambda\sum_{i=1}^{p-1}\max\{\xx_i - \xx_{i+1}\}\\
  \equiv &\minimize_{\xx\in \RR^p}\underbrace{\frac{1}{n}\sum_{i=1}^n \log(1 + \exp(- b_i \boldsymbol{a}_i^T \xx))}_{=f(\xx)} + \underbrace{\lambda\sum_{i=1}^{\lfloor p/2\rfloor }\max\{\xx_{2i} - \xx_{2i+1}\}}_{=g(\xx)}
  + \underbrace{\lambda\sum_{i=1}^{\lfloor (p-1)/2\rfloor }\max\{\xx_{2i+1} - \xx_{2i+2}\}}_{=h(\xx)}~,
\end{align}
where $\lambda$ is a regularization parameter. The problem dimension in this case is $n=100$, $p=50$, and an noise of variance of 5 (instead of 1 as described in \ref{apx:synthetic}). We increased the noise variance to make the estimation of the ordering of coefficients more challenging (see figure below).

\paragraph{Benchmarks.} We show recovered coefficients as well as the time benchmarks in Figure~\ref{fig:faces}. The Adaptive TOS method is the best performing method in all four cases, and the difference is larger in the low regularization setting.

\paragraph{Lipschitz constant of proximal term.} Let $\|\cdot\|_{\text{ISO}}$ denote the nearly isotonic pseudo-norm. Then we can bound it in terms of the euclidean norm as follows, where as in \ref{apx:deblurring}, $\boldsymbol{K}$ is the matrix of first order differences:
\begin{equation}
    \|\xx\|_{\text{ISO}} \leq \|\boldsymbol{K}\xx\|_1 \leq \|\boldsymbol{K}\|_1 \|\xx\|_1 = 2\|\xx\|_1 \leq 2\sqrt{p}\|\xx\|
\end{equation}

\begin{figure*}[h]
\includegraphics[width=\linewidth]{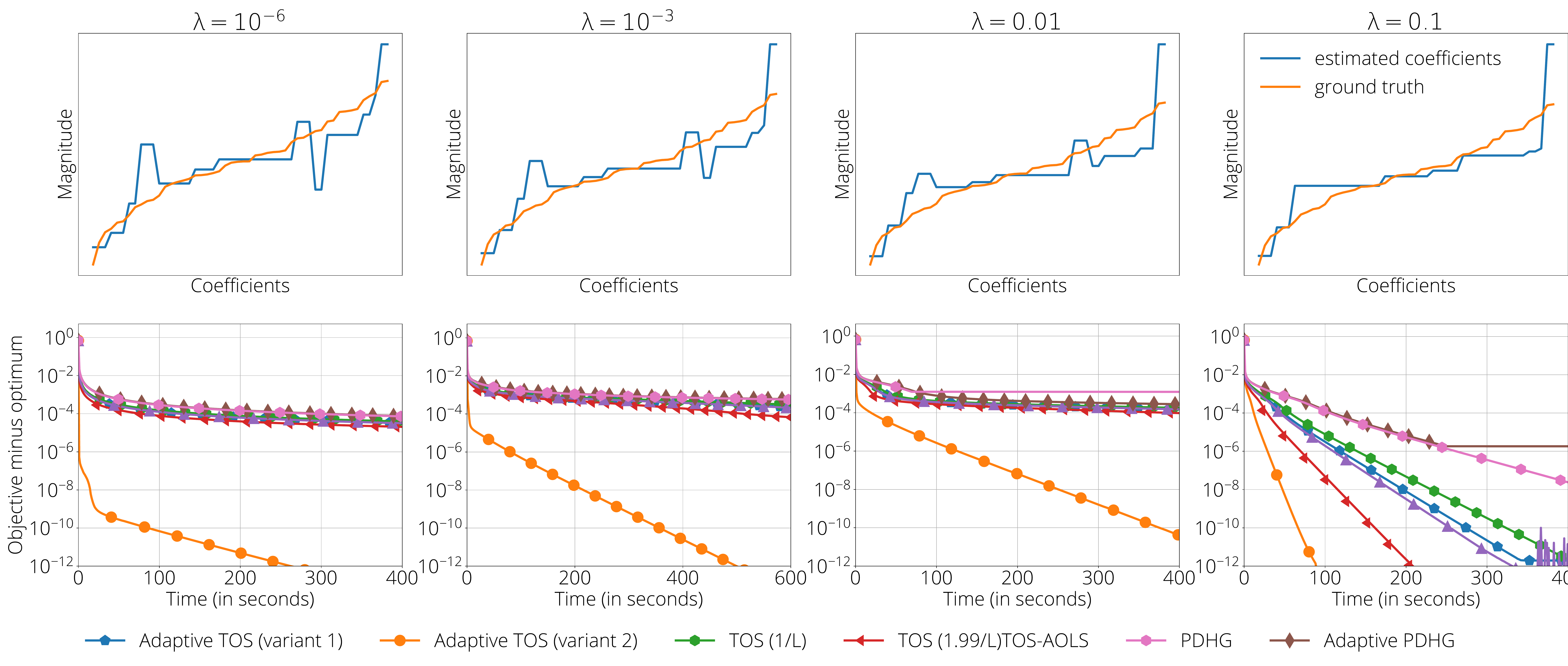}
\caption{{\bfseries Logistic regression with nearly isotonic penalty}. Top row: ground truth and recovered coefficients. Bottom row: time vs suboptimality comparison. Columns represent different values of the $\lambda$ regularization parameter}\label{fig:isotonic}
\end{figure*}

\end{document}